\documentclass[12pt]{amsart}

\usepackage{tikz}
\usepackage{bbm}
\usetikzlibrary{3d,arrows,calc,positioning,decorations.pathreplacing,matrix} 

\usepackage{calligra,mathrsfs}
\usepackage[all]{xy}
\usepackage{float, comment}
\usepackage{mathtools}
\usepackage{amsmath}
\usepackage{amsthm}
\usepackage{amssymb}
\usepackage{amsbsy}
\usepackage{amstext} 
\usepackage{amsopn}
\usepackage[mathscr]{eucal}
\usepackage{enumerate}
\usepackage{xcolor}
\usepackage{graphicx} 
\usepackage{scalerel}
\usepackage{microtype} 
\usepackage[margin=1in,marginparwidth=0.8in, marginparsep=0.1in]{geometry}
\usepackage[pagebackref, bookmarks=true, bookmarksopen=true, bookmarksdepth=3,bookmarksopenlevel=2, colorlinks=true, linkcolor=blue, citecolor=blue, filecolor=blue, menucolor=blue, urlcolor=blue]{hyperref}

\usepackage{tikz}
\usepackage{bbm}
\usepackage[all]{xy}
\usetikzlibrary{arrows,calc,positioning,decorations.pathreplacing} 

\numberwithin{equation}{section}
\newtheorem{Theorem}[equation]{Theorem}
\newtheorem{Proposition}[equation]{Proposition} 
\newtheorem{Lemma}[equation]{Lemma}

\newtheorem{Conjecture}[equation]{Conjecture}

\theoremstyle{definition}
\newtheorem{Remark}[equation]{Remark}
\newtheorem{Example}[equation]{Example}
\newtheorem{Definition}[equation]{Definition}

\numberwithin{figure}{section}

\newcommand{\bC}{\mathbb{C}}
\newcommand{\bD}{\mathbb{D}}

\newcommand{\bN}{\mathbb{N}}

\newcommand{\bR}{\mathbb{R}}

\newcommand{\bZ}{\mathbb{Z}}

\newcommand{\cE}{\mathcal{E}}
\newcommand{\cF}{\mathcal{F}}
\newcommand{\cG}{\mathcal{G}}
\newcommand{\cH}{\mathcal{H}}

\newcommand{\cQ}{\mathcal{Q}}

\newcommand{\sW}{\mathscr{W}}

\newcommand{\al}{\alpha}
\newcommand{\be}{\beta}

\newcommand{\ep}{\epsilon}
\newcommand{\ga}{\gamma}

\newcommand{\ol}{\overline}

\newcommand{\into}{\hookrightarrow}

\newcommand{\congto}{\xrightarrow{\sim}}

\newcommand{\id}{{id}}
\newcommand{\pt}{\mathrm{pt}}
\newcommand{\op}{\mathrm{op}}
\newcommand{\sgn}{\mathrm{sgn}}

\DeclareMathOperator{\Hom}{Hom}
\DeclareMathOperator{\End}{End}
\DeclareMathOperator{\Ext}{Ext}

\DeclareMathOperator{\Arg}{Arg}
\DeclareMathOperator{\Log}{Log}

\newcommand{\Coh}{\mathrm{Coh}}
\newcommand{\conv}{{\mathbin{\scalebox{1.1}{$\mspace{1.5mu}*\mspace{1.5mu}$}}}}

\newcommand{\Int}{\mathrm{Int}}

\newcommand{\Loc}{\mathrm{Loc}}
\newcommand{\Mod}{\mathrm{Mod}}
\newcommand{\Rep}{\mathrm{Rep}}

\newcommand{\Sh}{\mathrm{Sh}}

\renewcommand{\wr}{W}
\newcommand{\wrneg}{W^-}

\newcommand{\incl}{\iota}
\newcommand{\inclarg}[1]{\iota_{{#1}}}
\newcommand{\pd}{\mathrm{pd}}
\newcommand{\supp}{\mathrm{supp}}

\newcommand{\opint}{(0,1)}
\newcommand{\clint}{[0,1)}
\newcommand{\Exp}{E}

\DeclareMathOperator{\cHom}{\mathscr{H}\text{\kern -3pt {\calligra\large om}}\,}

\DeclareFontFamily{U}{mathx}{\hyphenchar\font45}
\DeclareFontShape{U}{mathx}{m}{n}{
	<5> <6> <7> <8> <9> <10>
	<10.95> <12> <14.4> <17.28> <20.74> <24.88>
	mathx10
}{}
\DeclareSymbolFont{mathx}{U}{mathx}{m}{n}
\DeclareFontSubstitution{U}{mathx}{m}{n}
\DeclareMathAccent{\widecheck}{0}{mathx}{"71}

\DeclareMathSymbol{\shortminus}{\mathbin}{AMSa}{"39}

\newcommand{\arrtip}{latex'}

\begin{document}
\title[Tropical Lagrangian coamoebae and generalized dimer models]{Tropical Lagrangian coamoebae and \\ generalized dimer models}

\author[Christopher Kuo]{Christopher Kuo}
\address[Christopher Kuo]{University of Southern California \\ Los Angeles CA, USA}
\email{chrislpkuo@berkeley.edu}

\author[Harold Williams]{Harold Williams}
\address[Harold Williams]{University of Southern California \\ Los Angeles CA, USA}
\email{hwilliams@usc.edu}

\begin{abstract}
	We study the coamoebae of Lagrangian submanifolds of $(\mathbb{C}^\times)^n$, specifically how the combinatorics of their degenerations encodes the homological algebra of mirror coherent sheaves. Concretely, to a minimal free resolution $F^\bullet$ of a module $M$ over $\mathbb{C}[z_1^{\pm 1}, \dotsc, z_n^{\pm 1}]$ we associate a simplicial complex $T(F^\bullet) \subset T^n$. We call $T(F^\bullet)$ a tropical Lagrangian coamoeba. We show that the discrete information in $F^\bullet$ can often be recovered from $T(F^\bullet)$, and that more generally $M$ is mirror to a certain constructible sheaf supported on~$T(F^\bullet)$. The resulting interplay between coherent sheaves on~$(\mathbb{C}^\times)^n$ and simplicial complexes in $T^n$ provides a higher-dimensional generalization of the spectral theory of dimer models in~$T^2$, as well as a symplectic counterpart to the theory of brane brick models. 
\end{abstract}

\dedicatory{To Vladimir Fock on the occasion of his 60th birthday.}

\maketitle

\setcounter{tocdepth}{1}

\tableofcontents

\section{Introduction}
\thispagestyle{empty}

Given a subvariety $Z \subset (\bC^\times)^n$, we may consider its images under the projections
\begin{align*}
	\Log: (\bC^\times)^n \to \bR^n, &\quad (z_1, \dotsc, z_n) \mapsto (\log |z_1|, \dotsc, \log |z_n|), \\
	\Arg: (\bC^\times)^n \to T^n, &\quad  (z_1, \dotsc, z_n) \mapsto (\Arg (z_1), \dotsc, \Arg (z_n)). 
\end{align*}
These are respectively referred to as its amoeba $A(Z)$ \cite{GKZ94} and coamoeba (or alga) $C(Z)$ \cite{Pas04,FHKV08}. Their geometric, analytic, and combinatorial properties have proven to be of interest from a wide range of perspectives. 

One can also study the amoebae and coamoebae of other subsets of $(\bC^\times)^n$, in particular Lagrangian submanifolds. Typically, one considers Lagrangians up to an equivalence such as Hamiltonian isotopy. Many different amoebae and coamoebae appear as we consider all Lagrangians in an equivalence class, and they can be arbitrarily large in general: if $L$ is noncompact, a large enough Hamiltonian isotopy will make its coamoeba all of $T^n$. But they cannot be arbitrarily small, and interesting combinatorics emerges when we consider Lagrangian degenerations whose amoebae or coamoebae become in some sense minimal. 

In the setting of amoebae, this combinatorics is that of tropical varieties. A tropical variety in $\bR^n$ can often be realized as the limit of the amoebae of a degenerating family of Lagrangians, and the questions of whether and how a given tropical variety can be realized this way have been of active interest \cite{Mik19,Hic19,MR20,Mat21,Hic22,Han24}. This interplay between symplectic and tropical geometry is rooted in mirror symmetry: roughly speaking, Lagrangian realizations of tropical varieties should be mirror to coherent sheaves which have the same tropicalization. 

In this paper we introduce and study a parallel class of objects which we call tropical Lagrangian coamoebae. See Figure \ref{fig:introexample} for an example. Our consideration of them is also rooted in mirror symmetry: a tropical Lagrangian coamoeba is a certain geometric representation of a free resolution of a mirror coherent sheaf. A succinct characterization of this construction, and a separate source of motivation, is that it generalizes the spectral theory of dimer models to accommodate arbitrary coherent sheaves on algebraic tori. In particular, a special case of the relationship between a tropical Lagrangian coamoeba and the corresponding free resolution is the relationship between a bipartite graph and its Kasteleyn matrix. 

\begin{figure}
	\centering
	\begin{tikzpicture}
		\node (img1) at (-.25,0) {\includegraphics[width=10cm]{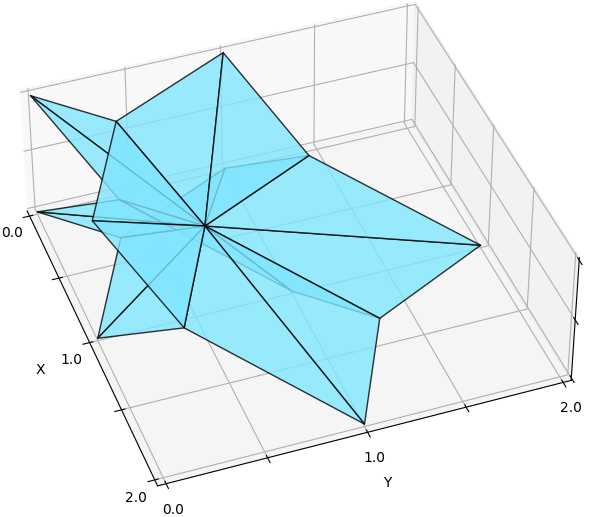}};  
		
		\node (img2) at (8.25,0) {\includegraphics[width=5.5cm]{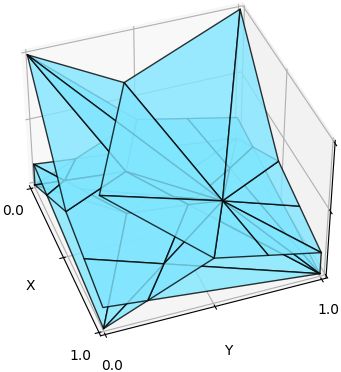}};  
		
		\draw[-stealth', thick] (img1.east) -- node[above] {} (img2.west);
	\end{tikzpicture}
	\caption{The ``origami'' on the right is a tropical Lagrangian coamoeba, pictured via its lift from $T^3$ to $[0,1]^3$. It is mirror to the Koszul resolution of the curve in $(\bC^\times)^3$ defined by $1 + xy + z = 0$ and $1 + x + y = 0$, and is ``folded'' out of the simplicial complex on the left; see Example \ref{ex:introexample} for details.}
	\label{fig:introexample}
\end{figure}

\subsection{Coamoebae from resolutions} Let $F^\bullet$ be a minimal-length free resolution of a finitely generated $R_n$-module $M$, where $R_n := \bC[z_1^{\pm 1},  \dotsc, z_n^{\pm 1}]$. We assume a homogeneous basis is chosen for $F^\bullet$ and a point $x_i \in \bR^n$ is chosen for each basis element. In Section \ref{sec:restoco} we associate to this data a simplicial complex $X(F^\bullet)$ with a map to $T^n$, and define the tropical Lagrangian coamoeba $T(F^\bullet)$ as its image. In fact, these constructions make sense for more general $F^\bullet$, but we suppress this in the introduction. 

The objects $X(F^\bullet)$ and $T(F^\bullet)$ are determined by the discrete information in $F^\bullet$, and our first results show that this information can often be recovered from them. More precisely, note that $F^\bullet$ and its basis are encoded by the sequence of $R_n$-valued matrices representing its differentials. By the discrete information in $F^\bullet$ we mean the sizes of these matrices and the exponents appearing in each entry --- but not the coefficients. In Theorem~\ref{thm:recovery} we show this data can be recovered from $X(F^\bullet)$, its map to $T^n$, and the degrees of its vertices (which correspond to basis elements of $F^\bullet$). We also show that under mild hypotheses $X(F^\bullet)$ is immersed (Proposition~\ref{prop:immersedsimpset}), and is determined by $T(F^\bullet)$ and the images of its vertices (Theorem \ref{thm:recovery2}). 

Our next results characterize the simplicial complexes that arise from this construction. A motivating case is when $F^\bullet$ has two terms and the resulting $T(F^\bullet)$ is a graph. The degrees of its vertices provide a bipartite coloring, and any bipartite graph with linear edges arises this way for some $F^\bullet$. More generally, in Proposition~\ref{prop:characterization} we show that if the 1-skeleton of a $k$-dimensional simplicial complex $X$ mapping to $T^n$ admits a vertex coloring by the degrees $\{\shortminus k, \dotsc, \shortminus 1, 0\}$, and if $X$ is then generated in a certain sense by the edges which increase degree by one, then $X$ is of the form $X(F^\bullet)$ for some sequence of free modules $F^\bullet$. 

\subsection{Mirror complexes} Our next claim is that in a suitable sense $T(F^\bullet)$ is the coamoeba of a degenerate Lagrangian brane which is mirror to $M$. We write $\Coh((\bC^\times)^n)$ for the (derived) category of coherent sheaves
on $(\bC^\times)^n$, or equivalently of finitely generated $R_n$-modules. Recall that mirror symmetry identifies this with $\sW(T^* T^n)$, the wrapped Fukaya category of $T^* T^n$. The basic objects of $\sW(T^* T^n)$ are Lagrangian branes, which here will mean eventually conical exact Lagrangians equipped with certain additional data. 
 
Lagrangian branes in a cotangent bundle $T^* X$ are closely related to sheaves on the base $X$ \cite{NZ09,GPS18}, a principle exploited extensively in the context of mirror symmetry \cite{Bon06,FLTZ11,SS16,STZ14,Nad19,Zho19,Kuw20,GN20,GS22}. A robust heuristic that emerges is that a constructible sheaf $\cF$ on $X$ can be thought of as a degenerate Lagrangian brane supported on its singular support $ss(\cF) \subset T^* X$. Recall that $ss(\cF)$ is a singular conical Lagrangian that measures the deviation of $\cF$ from being locally constant~\cite{KS94}. 

In particular, for compact $X$ there is a functor from the (derived) category $\Sh^b(X)$ of constructible sheaves to $\sW(T^* X)$. First, we apply the left adjoint $W: \Sh(X) \to \Loc(X)$ of the inclusion $\Loc(X) \subset \Sh(X)$ of local systems into all sheaves. This takes constructible sheaves to the compact subcategory $\Loc^c(X) \subset \Loc(X)$ \cite{Kuo23}, which is in turn equivalent to $\sW(T^* X)$ \cite{Abo11}. Conversely, a Lagrangian brane $L$ may be quantized to a sheaf $Q(L)$ whose singular support is contained in the conical Lagrangian $\lim_{t \to 0} t L$ \cite{Gui16,JT17}. Heuristically, $Q(L)$ is the degeneration of the branes $t L$ as $t \to 0$, and indeed $L \cong W Q(L)$ in $\sW(T^* X)$.  Note that the projection of $ss(\cF)$ to $X$ is the ordinary support of $\cF$, hence the support of a sheaf is the analogue of the coamoeba of a Lagrangian brane. 

With this context in mind, our precise claim (Theorem~\ref{thm:mirrorcomplex}) is that there exists a constructible sheaf which is supported on $T(F^\bullet)$, and which becomes isomorphic to $M$ under 
$$ \Sh^b(T^n) \xrightarrow{\wr} \Loc^c(T^n) \congto \Coh((\bC^\times)^n),$$ 
where on the right we take a local system to its monodromy representation. 
This sheaf is represented by a complex $C^\bullet(F^\bullet)$ which is a sum of sheaves of the form $\pi_* \bC_{S_i}$. 
Here $\pi: \bR^n \to T^n$ is the projection, and the $S_i \subset \bR^n$ are certain contractible simplicial complexes labeled by the basis of $F^{\bullet}$. 
If $i$ labels a basis element in degree $\shortminus k$, then $S_i$ is at most $k$-dimensional and depends only on the truncation of $F^\bullet$ to degrees $\geq \shortminus k$. Non-isomorphic resolutions of $M$ generally lead to mirror complexes which are not even quasi-isomorphic, so (unlike $F^\bullet$) the isomorphism class of $C^\bullet(F^\bullet)$ in the derived category retains all information of interest; see Figure \ref{fig:introdiagram} for a schematic. 

The principle of this construction is the following. If the differentials in $F^\bullet$ were zero it would be mirror to a direct sum of skyscraper sheaves in different degrees. To realize the differentials as morphisms of sheaves, we must ``stretch out'' the supports of these skyscrapers so they can interact. Working out the minimal amount of stretching required leads to an inductive characterization of the $S_i$ --- and thus of $T(F^\bullet)$, which is recovered as the union of their projections to $T^n$ --- together with the desired sheaves (Proposition~\ref{prop:SiofF}). 

\begin{figure}
	\begin{tikzpicture}
		[baseline=(current  bounding  box.center),thick,>=\arrtip]
		\node (a) at (0,0) {$\mathrm{Ch}^b(\mathrm{Free}_{R_n}^{fg})$};
		\node (b) at (4.5,0) {$\Sh^b(T^n)$};
		\node (c) at (0,-1.5) {$\Coh((\bC^\times)^n)$};
		\node (d) at (4.5,-1.5) {$\Loc^c(T^n)$};
		\draw[->,dashed] (a) to node[above] {$C^\bullet(-)$} (b);
		\draw[->] (b) to node[right,pos=.4] {$\wr$} (d);
		\draw[->] (a) to node[left] {$ $}(c);
		\draw[->] (c) to node[above] {$\sim$} (d);
	\end{tikzpicture}
	\caption{
		The (derived) categories of local systems and coherent sheaves are localizations of the categories of constructible sheaves and bounded complexes of finitely generated free modules, respectively. The mirror equivalence on bottom is lifted by the assignment $F^\bullet \mapsto C^\bullet(F^\bullet)$ to a function between objects of these finer categories.}
	\label{fig:introdiagram}
\end{figure}

\subsection{Realizability}
Next we consider whether $T(F^\bullet)$ can be realized as a degeneration of coamoebae of smooth Lagrangian branes mirror to $M$. This question is analogous to that of the Lagrangian realizability of tropical varieties mentioned earlier. In Theorem \ref{thm:hypersmoothing} we show such a realization exists when $F^\bullet$ is a two-term complex and $T(F^\bullet)$ is embedded, and that moreover the sheaf quantizations of these branes converge to $C^\bullet(F^\bullet)$ in the relevant sense. 

The Lagrangians appearing in this result include a number of examples in the existing literature \cite{STWZ19,Mat21,Hic19}. When $n=2$ their sheaf quantizations were termed alternating sheaves in \cite{STWZ19}.  In Section \ref{sec:branes} we extend this terminology and a number of results of \cite{STWZ19} to higher dimensions. We then consider degenerations of alternating sheaves in which their support retracts onto a bipartite graph. In Proposition \ref{prop:limitaltsheafquant} we show that the limit (i.e. real nearby cycles) of such a degeneration is a ``reflected local system'', a certain constructible sheaf related to a graph local system by reflection functors. This provides a key ingredient in the proof of Theorem \ref{thm:hypersmoothing}, as we also show that in the two-term case $C^\bullet(F^\bullet)$ is quasi-isomorphic to a reflected local system in degree $\shortminus 1$. 

It is less clear what predictions to make in general. While we assume $T(F^\bullet)$ can always be realized as a degeneration of coamoebae of Lagrangian branes less singular than the stratified conormal bundle $N^* T(F^\bullet)$, it is evident that some singular behavior must be allowed. Indeed, typically $T(F^\bullet)$ is itself only immersed in the relevant sense. In this case we can expect at best to realize it as a degeneration of coamoebae of immersed Lagrangians. 

We can imagine that there exists a construction which begins by approximating the $S_i$ with coamoebae of Hamiltonian isotopies of cotangent fibers, out of which we can then assemble a partial smoothing of $N^* T(F^\bullet)$. For example, the embedded Lagrangians we consider in the two-term case are, up to some modifications near infinity, the result of performing surgeries at isolated intersection points of isotoped cotangent fibers. But already once $T(F^\bullet)$ is two-dimensional, as in Figure~\ref{fig:introexample}, we do not know how to formulate such a prescription. As a special case, such partial smoothings should include the singular Lagrangian of \cite{Han24}, discussed below in Example \ref{ex:introlineexample}. 

\subsection{Minimality}

We have implicitly suggested that the tropical Lagrangian coamoebae associated to minimal-length free resolutions of $M$ are in some sense minimal among coamoebae of mirrors to $M$ (whether in the form of constructible sheaves or Lagrangian branes). To make this precise we propose the following. 
\begin{Conjecture}\label{conj:minimality}
If a constructible sheaf $\cF \in \Sh^b(T^n)$ is mirror to an $R_n$-module~$M$, then the dimension of the support of $\cF$ is at least the projective dimension of $M$. 
\end{Conjecture}
\noindent One could also formulate a more geometric variant: if $\{L_t\}_{t>0}$ is a family of Lagrangian branes mirror to $M$, then the dimension of $\lim_{t \to 0} C(L_t)$ is at least $\pd(M)$. As evidence for the conjecture, we prove in Theorem \ref{thm:minimality} that it holds under the further hypothesis that $\cF$ is concentrated in a single degree. 

By construction $ \dim T(F^\bullet) = \pd(M)$, hence the conjecture would imply that $T(F^\bullet)$ has minimal dimension among mirror coamoebae of $M$. In particular, this immediately suggests an axiomatic notion of tropical Lagrangian coamoeba: any subset of $T^n$ which arises as the support of a constructible sheaf $\cF$ whose support has minimal dimension given the isomorphism class of $W\cF$ (equivalently, of its mirror $R_n$-module). In this framework, the role of the present paper would be to provide a source of explicit, piecewise-linear examples, and a natural question would be how far these are from providing a complete list of tropical Lagrangian coamoebae up to isotopy. Note also that this would provide a notion of tropical Lagrangian coamoeba in an arbitrary manifold $X$, since the above condition makes sense when $\cF$ is a sheaf on $X$. 

\subsection{Context: dimer models}
A major source of our interest in the constructions we introduce is that they reinterpret the spectral theory of dimer models as a special case of a general higher-dimensional picture.  Recall that dimer models deal with perfect matchings on bipartite graphs. For an edge-weighted graph $\Gamma \subset T^2$ (and the associated periodic graph in $\bR^2$), the dimer model is solvable in terms of a matrix-valued Laurent polynomial, the Kasteleyn operator $K(z,w)$ \cite{Kas67, KOS06}. 

The spectral curve $\det K(z,w) = 0$ in particular plays a key role in this theory. This curve is the support of the spectral data of $K(z,w)$ --- the coherent sheaf corresponding to its cokernel --- from which the edge weights on $\Gamma$ can be recovered up to gauge equivalence. The resulting interplay between graphical combinatorics and spectral data turns out to be of interest from a wide range of perspectives, including real algebraic geometry \cite{KO06}, string theory \cite{FHKV08}, integrable systems and cluster algebras \cite{GK13,FM16}, and of course probability theory \cite{KO07}. 

In \cite{TWZ19} it was shown that the passage to spectral data is a mirror map: an edge-weighted graph $\Gamma \subset T^2$ determines, via the construction of \cite{STWZ19}, a Lagrangian brane in $T^* T^2$ whose coamoeba is approximated by $\Gamma$, and whose mirror coherent sheaf is the spectral data of $K(z,w)$. But from the perspective of this coherent sheaf, the matrix $K(z,w)$ is the data of a free resolution, and the graph $\Gamma$ is the associated tropical Lagrangian coamoeba. One summary of the present paper is thus that the constructions just recalled admit a robust generalization, where $K(z,w)$ is replaced by any free resolution of any coherent sheaf on an algebraic torus of any dimension, and where the graph $\Gamma$ is replaced by a more general simplicial complex.  

\subsection{Context: brane brick models} Our constructions also provide a symplectic counterpart to the tropical coamoebae or brane brick models of \cite{FU14,FGLSY15,FLS16,FLSV17}. These generalize dimer models in $T^2$ in a different direction, extending two key observations of \cite{FHKV08}. First, the coamoeba of a sufficiently nice curve $\Sigma \subset (\bC^\times)^2$ retracts onto a bipartite graph $\Gamma$ of which $\Sigma$ is a spectral curve. And second, the adjacencies of the faces of $\Gamma$ determine the intersections of vanishing cycles of the equation defining~$\Sigma$. The works above argue more generally that the coamoeba of a sufficiently nice hypersurface in $(\bC^\times)^n$ retracts onto a codimension-one polyhedral complex (called a tropical coamoeba in \cite{FU14} or, when $n=3$, a brane brick model \cite{FGLSY15}), and that this complex divides $T^n$ into chambers whose adjacencies continue to encode intersections of vanishing cycles. 

We thus have two different generalizations of dimer models in $T^2$, the one just recalled and the one we provide. But this is to be expected: a curve in $(\bC^\times)^2$ is both a complex hypersurface and --- by applying a hyperk\"ahler rotation --- a Lagrangian submanifold. In particular, bipartite graphs in~$T^2$ are degenerations of both hypersurface coamoebae and Lagrangian coamoebae. In higher dimensions, however, these classes of coamoebae no longer overlap. Studying their combinatorial approximations thus leads to two generalizations of bipartite graphs in $T^2$, each extending different features of the dimer model. The one studied in \cite{FU14,FGLSY15} and successors extends the way adjacencies of faces of graphs in $T^2$ encode intersections of vanishing cycles. The one studied here instead extends the way edge weightings on graphs in~$T^2$ parametrize spectral data. 

Let us also mention that the coamoebae of subvarieties of $(\bC^\times)^n$ have been studied from various other points of view besides those above. We refer the reader to \cite{Nis09,NS13a,NS13b} for other studies of their combinatorial properties, and to \cite{KZ18,KN21} for connections to the topology of complex hypersurfaces. 

\subsection{Examples}\label{sec:introexamples}

We close the introduction with a few examples of our constructions.

\begin{Example}\label{ex:introexample}
	Consider the curve in $(\bC^\times)^3$ defined by $f = g = 0$, where $f = 1 + x + y$ and $g = 1 + z + xy$. Take $F^\bullet$ to be the Koszul resolution of its coordinate ring $R_3/(f,g)$:
	\begin{equation}\label{eq:Koszul}
		\begin{tikzpicture}
			[baseline=(current  bounding  box.center),thick,>=\arrtip]
			\node (a) at (0,0) {$R_3$};
			\node (b) at (3.5,0) {$R_3 \oplus R_3$};
			\node (c) at (7,0) {$R_3$};
			\draw[->] (a) to node[above] {$\begin{bmatrix}
					\shortminus g \\
					f
				\end{bmatrix} $} (b);
			\draw[->] (b) to node[above] {$\begin{bmatrix}
					f & g 
				\end{bmatrix}$} (c);
		\end{tikzpicture}
	\end{equation}
	The right picture in Figure \ref{fig:introexample} illustrates the associated tropical Lagrangian coamoeba~$T(F^\bullet)$. More precisely, $T(F^\bullet)$ depends on the choice of a point $x_i \in \bR^3$ for each rank-one summand of $F^\bullet$, and here we take $x_1 = (\frac23, \frac23, \frac13)$, $x_2 = (\frac13, \frac13, \frac13)$, $x_3 = (\frac13, \frac13, 0)$, and $x_4 = (0,0,0)$ (enumerating the summands above from left to right). 
	
	This coamoeba $T(F^\bullet)$ is the image of eighteen 2-simplices in $\bR^3$. These correspond to the eighteen terms that appear while checking that $d^2 = 0$, before performing any cancellations:
	\begin{equation*}
		\begin{bmatrix}
			1 + x + y & 1 + z + xy 
		\end{bmatrix}
		\begin{bmatrix}
			\shortminus(1 + z + xy) \\
			1 + x + y
		\end{bmatrix}
		=
		\begin{aligned}
			&\shortminus(1 + z + xy + x + xz + x^2y + y + yz + xy^2) \\
			&\,\, + (1 + x + y + z + xz + yz + xy + x^2y + xy^2).
		\end{aligned}
	\end{equation*}
	Specifically, to each term on the right we associate a 2-simplex in $\bR^3$, one of whose vertices is the exponent of that term. Its other vertices are $x_1$ and an integer translate of either $x_2$ or $x_3$, depending on whether the term came from $\shortminus g \cdot f$ or $f \cdot g$ (in which case the translation is by the exponent of the relevant term in $\shortminus g$ or $f$, respectively). 
	
	The union $S_1 \subset \bR^3$ of these simplices is pictured on the left of Figure \ref{fig:introexample}. It is the cone over the complete bipartite graph $K_{3,3}$ whose vertices are the translates of $x_2$ and $x_3$ indicated above (with each edge kinked at an additional bivalent vertex corresponding to an exponent of one of the right-hand terms). If we replaced $f$ and $g$ with any other trinomials, $T(F^\bullet)$ would still be the projection of the cone over some $K_{3,3}$, but the map $K_{3,3} \to \bR^3$ would change so as to encode the new exponents in $f$ and $g$. On the other hand, changing only the coefficients in $f$ and $g$ would not affect the picture. 
	
	More generally, if we replaced $f$ and $g$ with arbitrary Laurent polynomials with $m$ and~$n$ terms, respectively, $T(F^\bullet)$ would be the projection of the cone over some $K_{m,n}$ (possibly mapped into $\bR^3$ with self-intersections). We also note that the 1-skeleton of $S_1$ (and its image in $T(F^\bullet)$) is naturally tripartite, with vertices colored by the degrees of the associated summands of $F^\bullet$.  
	
	The mirror complex $C^\bullet(F^\bullet)$ has as its degree $\shortminus 2$ term the sheaf $\pi_* \bC_{S_1}$. The degree~$\shortminus 1$ term is $\pi_* \bC_{S_2} \oplus \pi_* \bC_{S_3}$, where $S_2$ and $S_3$ are certain graphs in the boundary of $S_1$. Specifically, $S_2$ is the union of the three segments connecting $x_2$ to the exponents of $f$, similarly for $S_3$ but instead with $x_3$ and $g$. Finally, the degree zero term is $\pi_* \bC_{S_4}$, where $S_4 = \{x_4\}$. 
	
	Varying the coefficients in $f$ and $g$, we obtain a family of curves in $(\bC^\times)^3$ parametrized by $(\bC^\times)^6$. The mirror complexes of the associated Koszul resolutions are a family of complexes of constructible sheaves supported on $T(F^\bullet)$, all built out of the same sheaves $\pi_* \bC_{S_i}$ but with varying differentials. Note also that each edge in $T(F^\bullet)$ is labeled by a term in $f$, $g$, or their product. Thus a choice of coefficients can be regarded as an edge-weighting on the 1-skeleton of $T(F^\bullet)$, just as in the setting of dimer models. 
\end{Example}

\begin{figure}
	\centering
	\begin{tikzpicture}
		\node (img1) at (-.25,0) {\includegraphics[width=8cm]{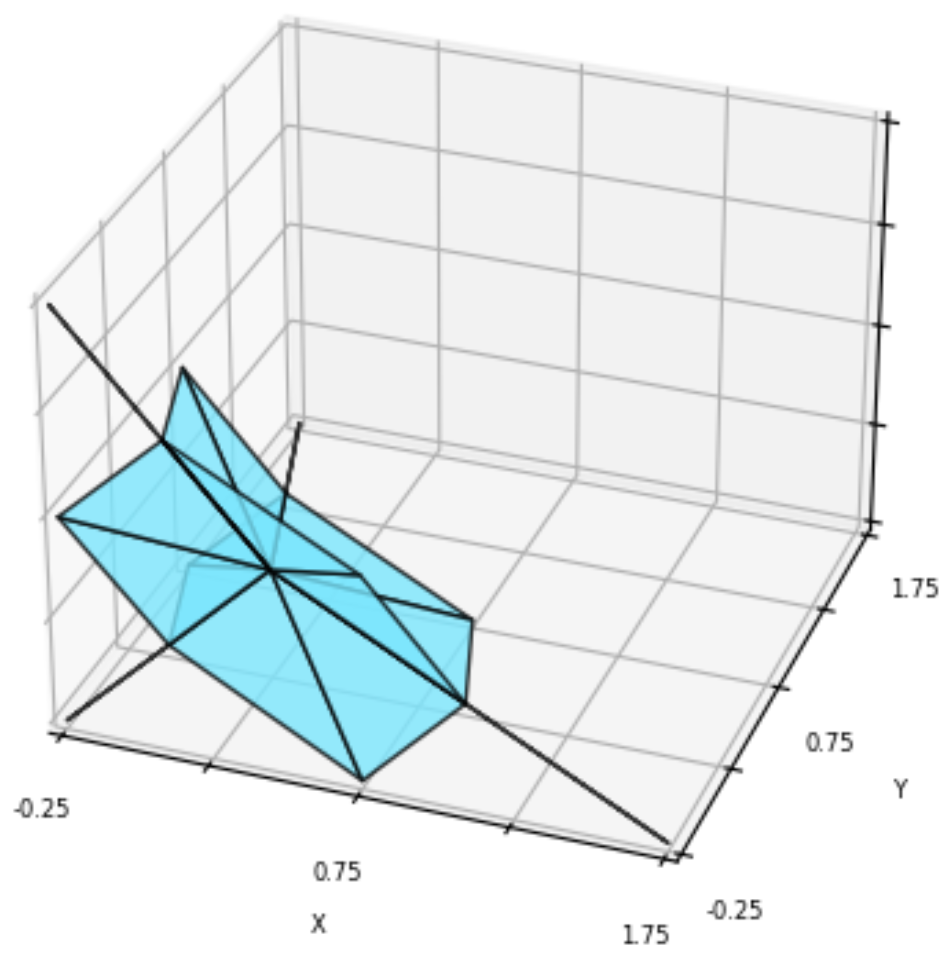}}; 
		
		\node (img2) at (8.25,0) {\includegraphics[width=6.5cm]{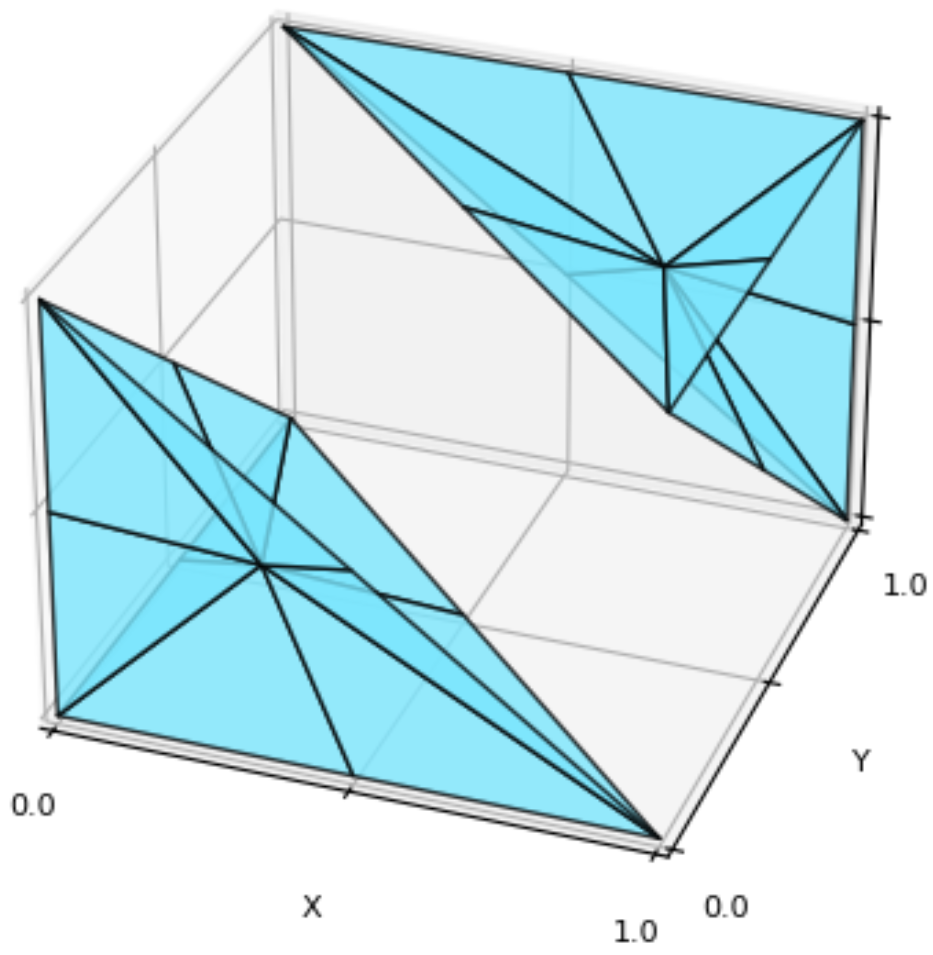}};  
		
		\draw[-stealth', thick] (img1.east) -- node[above] {} (img2.west);
	\end{tikzpicture}
	\caption{
		A tropical Lagrangian coamoeba mirror to the Koszul resolution of a line in $(\bC^\times)^3$. Concretely, it is the result of gluing two tetrahedra in~$T^3$ along their edges, then retracting each tetrahedron to a cone over its 1-skeleton; see Example \ref{ex:introlineexample}.}
	\label{fig:introlineexample}
\end{figure}

\begin{Example}\label{ex:introlineexample}
A generic line in $(\bC^\times)^3$ can be written as $f = g = 0$, where $f = a_0 + a_1x + a_2y + a_3z$ and $g = b_0 + b_1 x + b_2 y + b_3 z$ for some $a_0, \dotsc, b_3 \in \bC^\times$. Again take $F^\bullet$ to be the Koszul resolution  (\ref{eq:Koszul}), ordering its summands left to right. Taking $x_1 = (\frac14, \frac14, \frac14)$, $x_2 = x_3 = (0,0,0)$, and $x_4 = (\shortminus \frac14, \shortminus \frac14, \shortminus \frac14)$, the associated $T(F^\bullet)$ is pictured on the right of Figure \ref{fig:introlineexample}. 

For any choice of $x_i$, the coamoeba $T(F^\bullet)$ is the image of thirty-two 2-simplices in $\bR^3$. Setting $x_2 = x_3$, however, leads to only sixteen of these being distinct. Choosing the $x_i$ to be colinear further causes four of these 2-simplices to be degenerate. The union $S_1 \subset \bR^3$ of these simplices for the given choice of $x_i$ is pictured on the left of Figure \ref{fig:introlineexample}. 

This example connects our results to those of \cite{Han24}, which constructs a Lagrangian realization $L \subset (\bC^\times)^3$ of a 4-valent tropical curve in $\bR^3$. The coamoeba of $L$ is the union of two tetrahedra (the projections of $[0,e_1,e_2,e_3]$ and $[0,\shortminus e_1,\shortminus e_2,\shortminus e_3]$ to $T^3$ glued along their edges. This coamoeba retracts onto the tropical Lagrangian coamoeba $T(F^\bullet)$ in an obvious way, the Lagrangian~$L$ degenerating into $N^* T(F^\bullet)$ in the process. We expect that this example fits into the framework of Section \ref{sec:branes}, i.e. that the sheaf $C^\bullet(F^\bullet)$ should be a limit of sheaves that quantize $L$ and this degenerating family. However, $L$ is singular (it is not even immersed), and we do not know enough about the sheaf quantization of singular Lagrangians to pursue such a result. 
\end{Example}

\begin{Example}\label{ex:intrographexample}
Let $n = 2$ and consider the matrix
\begin{equation*}
K = \begin{bmatrix} a_1 + a_2 x & b_1 + b_2 y^{-1} \\ c_1 + c_2 y & d_1 + d_2 x^{-1} \end{bmatrix},
\end{equation*}
where $a_1, \dotsc, d_2 \in \bC^\times$. The complex $F^\bullet = R_2^2 \xrightarrow{K} R_2^2$ is a free resolution of its cokernel, which for generic coefficients is the pushforward of an invertible sheaf on an elliptic curve in~$(\bC^\times)^2$. The tropical Lagrangian coamoeba $T(F^\bullet)$ is the bipartite graph $\Gamma$ on the right of Figure~\ref{fig:intrographexample}, from which we can recover $K$ as the Kasteleyn matrix of the indicated edge weighting (up to signs). In particular, the cokernel of $K$ is the spectral data of the dimer model on $\Gamma$. Specifically, here we have chosen $x_1 = (\frac34, \frac34)$, $x_2 = (\frac14, \frac14)$, $x_3 = (\frac14,  \frac34)$, and $x_4 = (\frac34, \frac14)$, where 1 and 2 index the degree $\shortminus 1$ summands of $F^\bullet$ (hence the columns of $K$) and 3 and 4 the degree zero summands of $F^\bullet$ (hence the rows of $K$). 

The complex $C^\bullet(F^\bullet)$ is of the form $\pi_* \bC_{S_1} \oplus \pi_* \bC_{S_2} \xrightarrow{d_K} \pi_* \bC_{\{x_3\}} \oplus \pi_* \bC_{\{x_4\}}$. Here $S_1$ is the plus-shaped graph on the left of Figure \ref{fig:intrographexample} consisting of $x_1$ and the edges adjacent to it, similarly for $S_2$. The differential is an epimorphism whose kernel may be described as follows. Write $\Gamma_{w \to b}$ and $\Gamma_{b \to w}$ for the quivers obtained by directing the edges of $\Gamma$ from white to black and black to white, respectively. The coefficients $a_1, \dotsc, d_2$ define a rank-one local system on $\Gamma$, which we may encode as a representation of $\Gamma_{b \to w}$ with dimension vector $(1,1,1,1)$. Applying a reflection functor at each white vertex turns this into a representation of $\Gamma_{w \to b}$ with dimension vector $(1,1,3,3)$. We may then encode this as a constructible sheaf on $\Gamma$ which is locally constant away from the white vertices --- it is this sheaf which is the kernel of $d_K$. On account of this description, we refer to such sheaves as ``reflected local systems". 

In Section \ref{sec:branes} we further show that $C^\bullet(F^\bullet)$ can be described as a limit of alternating sheaves supported on the white and black regions cut out by the zig-zag paths of $\Gamma$. This connects our results to those of \cite{TWZ19}, where we showed that these alternating sheaves are mirror to the spectral data associated to $K$. 
\end{Example}

\begin{figure}
	\centering
	\begin{tikzpicture}
		\newcommand*{\scl}{3}; 
		\node (a) [matrix] at (0,0) {
			\draw[gray!20, thick] (0,.5*\scl)--(2*\scl,.5*\scl);
			\draw[gray!20, thick] (0,1.5*\scl)--(2*\scl,1.5*\scl);
			\draw[gray!20, thick] (.5*\scl,0)--(.5*\scl,2*\scl);
			\draw[gray!20, thick] (1.5*\scl,0)--(1.5*\scl,2*\scl);
			\draw[thick] (0,0)--(2*\scl,0)--(2*\scl,2*\scl)--(0,2*\scl)--(0,0);
			\draw (1.25*\scl,.75*\scl)--(1.25*\scl,1.75*\scl);
			\draw (.75*\scl,.25*\scl)--(.75*\scl,1.25*\scl);
			\draw (.25*\scl,.75*\scl)--(1.25*\scl,.75*\scl);
			\draw (.75*\scl,1.25*\scl)--(1.75*\scl,1.25*\scl);
			\node at (.65*\scl,.65*\scl) {$x_2$};
			\node at (1.35*\scl,1.35*\scl) {$x_1$};
			\node at (1.35*\scl,.65*\scl) {$x_4$};
			\node at (.65*\scl,1.35*\scl) {$x_3$};
			\draw[black,fill=white] (.25*\scl,.75*\scl) circle (.1cm);
			\draw[black,fill=white] (.75*\scl,.25*\scl) circle (.1cm);
			\draw[black,fill=white] (.75*\scl,1.25*\scl) circle (.1cm);
			\draw[black,fill=white] (1.25*\scl,.75*\scl) circle (.1cm);
			\draw[black,fill=white] (1.25*\scl,1.75*\scl) circle (.1cm);
			\draw[black,fill=white] (1.75*\scl,1.25*\scl) circle (.1cm);
			\draw[black,fill] (1.25*\scl,1.25*\scl) circle (.1cm);
			\draw[black,fill] (.75*\scl,.75*\scl) circle (.1cm);
			\\};
		\node (b) [matrix] at (7.5,0) {
			\draw[thick] (0,0)--(\scl,0)--(\scl,\scl)--(0,\scl)--(0,0);
			\draw (.75*\scl,0)--(.75*\scl,\scl);
			\draw (.25*\scl,0)--(.25*\scl,\scl);
			\draw (0,.75*\scl)--(\scl,.75*\scl);
			\draw (0,.25*\scl)--(\scl,.25*\scl);
			\node at (1.1*\scl,.75*\scl) {$a_2$};
			\node at (1.1*\scl,.25*\scl) {$d_2$};
			\node at (.75*\scl,1.1*\scl) {$c_2$};
			\node at (.25*\scl,1.1*\scl) {$b_2$};
			\node at (.5*\scl,.85*\scl) {$a_1$};
			\node at (.5*\scl,.15*\scl) {$d_1$};
			\node at (.85*\scl,.5*\scl) {$c_1$};
			\node at (.35*\scl,.5*\scl) {$b_1$};
			\node at (.75*\scl,-.1*\scl) {$\,$};
			\draw[black,fill] (.25*\scl,.25*\scl) circle (.1cm);
			\draw[black,fill] (.75*\scl,.75*\scl) circle (.1cm);
			\draw[black,fill=white] (.25*\scl,.75*\scl) circle (.1cm);
			\draw[black,fill=white] (.75*\scl,.25*\scl) circle (.1cm);
			\\};
		\draw[-stealth', thick] ($(a.east) + (0.5, 0)$) -- node[above] {} ($(b.west) + (-0.5, 0)$); 
	\end{tikzpicture}
	\caption{
		A bipartite graph in $T^2$ realized as the tropical Lagrangian coamoeba of its Kasteleyn matrix, interpreted as the differential in a two-term free resolution of the associated spectral data; see Example \ref{ex:intrographexample}.} 
	\label{fig:intrographexample}
\end{figure}
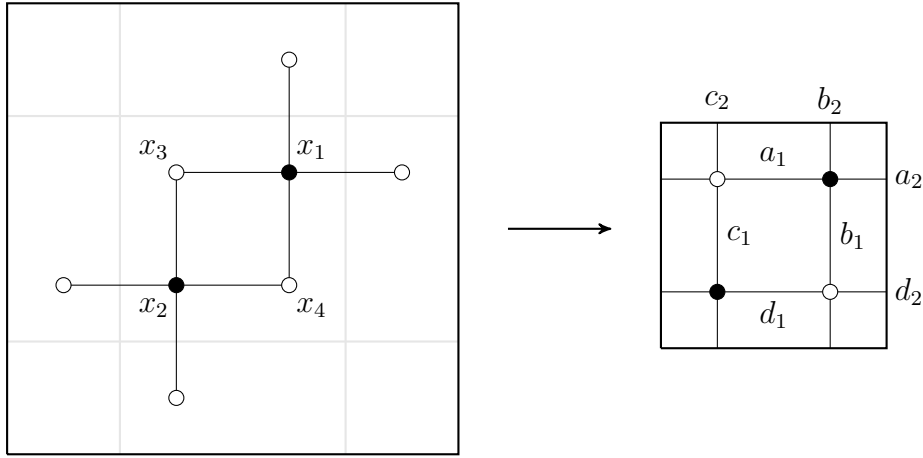

\begin{Example}
	Let $M \subset \bZ^n$ be a finite subset with $k$ elements, and let $f = \sum_{m \in M} c_m z^m$ be a Laurent polynomial with $c_m \neq 0$ for $m \in M$. The coordinate ring $R_n/(f)$ of the hypersurface $V(f) \subset (\bC^\times)^n$ has the defining resolution $F^\bullet = R_n \xrightarrow{f} R_n$. Taking $x_2$ to be the origin and $x_1$ to be generic, the tropical Lagrangian coamoeba $T(F^\bullet)$ is the image of the graph $S_1 \subset \bR^n$ consisting of $k$ edges radiating outward from $x_1$ to the elements of $M$. If $n>2$ this image is still an embedded graph, but if $n = 2$ its edges may cross each other. The complex $C^\bullet(F^\bullet)$ is of the form $\pi_* \bC_{S_1} \to \pi_* \bC_{\{0\}}$, and as in Example \ref{ex:intrographexample} it is quasi-isomorphic to a reflected local system placed in degree $\shortminus 1$. 
	
	When $n = 3$, this example connects our results to those of \cite{Lee07,LLP07}. For example, the graphs in Figures 8 and 9 of \cite{LLP07} are the $T(F^\bullet)$ just described when $f = 1 + x + y + z$ and $f = x + y + z + xy + xz + yz$, respectively. Similarly, the graphs in Figures 10 and 11 are the tropical Lagrangian coamoebae of rank-two resolutions of other hypersurfaces in $(\bC^\times)^3$. 
	
	In physical terms, a key theme of \cite{Lee07,LLP07} is that the use of dimer models in \cite{FHKV08} to analyze D3 branes probing $\mathrm{CY}_3$ singularities can be adapted to the setting of M2 branes probing $\mathrm{CY}_4$ singularities, but now with dimer models in $T^3$ rather than~$T^2$.  In particular, it is suggested in \cite[Sec. 3]{LLP07} that the topology of the special Lagrangian obtained from $V(f)$ by T-duality is described by ``twisting'' a tubular neighborhood of a corresponding bipartite graph in $T^3$ at each edge. Our results in Section~\ref{sec:branes} confirm this picture in a variant form: we embed this twisted tubular neighborhood as an exact ``alternating'' Lagrangian $L \subset T^* T^3$, and show that with suitable brane data $L$ is homologically mirror to the structure sheaf of $V(f)$. 
\end{Example}

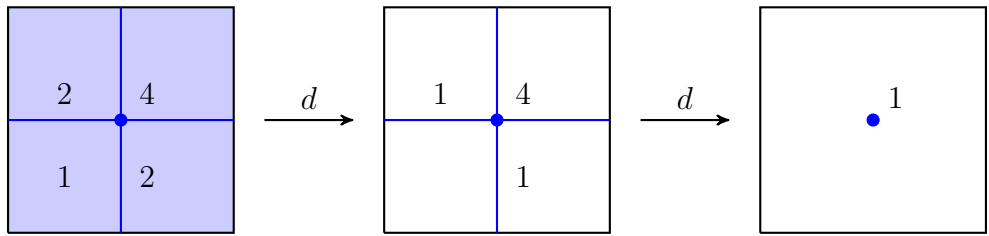
\begin{figure}
	\centering
	\begin{tikzpicture}
		\newcommand*{\scl}{3}; 
		\newcommand*{\spc}{5}; 
		\newcommand*{\rad}{.08cm}; 
		\node (a) [matrix] at (0,0) {
			\fill[blue!20] (0,0) rectangle (\scl,\scl);
			\draw[thick] (0,0)--(\scl,0)--(\scl,\scl)--(0,\scl)--(0,0);
			\draw[blue,thick] (0,.5*\scl)--(\scl,.5*\scl);
			\draw[blue,thick] (.5*\scl,0)--(.5*\scl,\scl);
			\draw[blue,fill] (.5*\scl,.5*\scl) circle (\rad);
			\node at (.5*\scl+.35,.5*\scl+.35) {$4$};
			\node at (.25*\scl,.5*\scl+.35) {$2$};
			\node at (.5*\scl+.35,.25*\scl) {$2$};
			\node at (.25*\scl,.25*\scl) {$1$};
			\\};
		\node (b) [matrix] at (\spc,0) {
			\draw[thick] (0,0)--(\scl,0)--(\scl,\scl)--(0,\scl)--(0,0);
			\draw[blue,thick] (0,.5*\scl)--(\scl,.5*\scl);
			\draw[blue,thick] (.5*\scl,0)--(.5*\scl,\scl);
			\draw[blue,fill] (.5*\scl,.5*\scl) circle (\rad);
			\node at (.5*\scl+.35,.5*\scl+.35) {$4$};
			\node at (.25*\scl,.5*\scl+.35) {$1$};
			\node at (.5*\scl+.35,.25*\scl) {$1$};
			\\};
		\node (c) [matrix] at (2*\spc,0) {
			\draw[thick] (0,0)--(\scl,0)--(\scl,\scl)--(0,\scl)--(0,0);
			\draw[blue,fill] (.5*\scl,.5*\scl) circle (\rad);
			\node at (.5*\scl+.3,.5*\scl+.3) {$1$};
			\\};
		\draw[-stealth', thick] ($(a.east) + (0.25, 0)$) -- node[above] {$d$} ($(b.west) + (-0.25, 0)$); 
		\draw[-stealth', thick] ($(b.east) + (0.25, 0)$) -- node[above] {$d$} ($(c.west) + (-0.25, 0)$); 
	\end{tikzpicture}
	\caption{When $F^\bullet$ is the Koszul resolution of a point, $C^\bullet(F^\bullet)$ is a complex of ``cubical'' sheaves quasi-isomorphic to the corresponding local system. Here the numbers in each picture indicate dimensions of stalks along the coordinate stratification of $T^2$; see Example \ref{ex:introcubicalexample}.} 
	\label{fig:introcubicalexample}
\end{figure}

\begin{Example}\label{ex:introcubicalexample}
	The most basic example of a pair of mirror sheaves is the skyscraper $\bC_a$ at $a = (a_1, \dotsc, a_n) \in (\bC^\times)^n$ and the local system $L_a$ with holonomies $a_1, \dotsc, a_n$ around the coordinate cycles in $T^n$. The skyscraper $\bC_a$ has a coordinate-wise Koszul resolution $F^\bullet$ with a homogeneous basis indexed by $\{0,1\}^n$, where $i = (i_1, \dotsc, i_n) \in \{0,1\}^n$ labels an element of degree $\shortminus \sum i_k$. Choose $x_{i}$ to have $k$th entry $\frac12$ if $i_k = 1$ and $0$ if $i_k = 0$. The associated simplicial complex $S_{(1,\dotsc,1)} \subset \bR^n$ is the hypercube $[0,1]^n$, expressed as a union of $2^n n!$ $n$-simplices. The tropical Lagrangian coamoeba $T(F^\bullet)$ is its image, which is all of $T^n$.  
	
	More generally, $S_i \subset \bR^n$ is the face of $S_{(1,\dotsc,1)}$ where we set the $k$th coordinate to be $0$ for each $k$ with $i_k = 0$. The mirror complex $C^\bullet(F^\bullet)$ is a ``resolution'' of the local system $L_a$ (shifted to degree $\shortminus n$) by direct sums of the ``cubical'' sheaves $\pi_* \bC_{S_i}$. Figure \ref{fig:introcubicalexample} illustrates this when $n=2$. Its three frames depict the three terms in $C^\bullet(F^\bullet)$ by indicating the dimensions of their stalks along the coordinate stratification of $T^2$. 
\end{Example}

\subsection*{Acknowledgements}
We thank Adam Boocher, Sebastian Franco, Sheel Ganatra, Peter Haine, Sebastian Haney, Jeff Hicks, David Nadler, Mahrud Sayrafi, Ralf Schiffler, David Treumann, Eric Zaslow, and Ilia Zharkov for valuable comments and discussions. We especially thank Wenyuan Li for helping us navigate many technical issues. C. K. was supported by NSF grant DMS-1928930, and H. W. was supported by NSF grant DMS-2143922. 

\section{Generalities} \label{sec: generalities}

We collect here some general results we will refer back to as needed, as well as summarizing the terminology and notation we will use. 

\subsection{Conventions and notation} By default, category will mean $\infty$-category, but the reader may safely ignore this (i.e. most arguments will apply equally to derived $\infty$-categories and their underlying triangulated categories). Given a ring $R$, we write $\Mod_R$ for the category of $R$-module spectra. Its homotopy category is the classical unbounded derived category of $R$-modules. The heart $\Mod_R^\heartsuit$ of its natural t-structure is the ordinary category of $R$-modules. 

By default space will mean locally compact Hausdorff space. We write $\Sh(X)$ for the category of $\Mod_\bC$-valued sheaves on a space $X$, and $\Loc(X)$ for its subcategory of locally constant sheaves. These categories have natural t-structures whose hearts $\Sh(X)^\heartsuit$ and $\Loc(X)^\heartsuit$ are the ordinary categories of $\Mod_\bC^{\heartsuit}$-valued sheaves and local systems. The category $\Loc(X)$ is compactly generated and we write $\Loc^c(X)$ for its subcategory of compact objects. 

To a continuous map $f: X \to Y$ are associated functors $f^*, f^!: \Sh(Y) \to \Sh(X)$ and $f_*,f_!: \Sh(X) \to \Sh(Y)$. Along with tensor and sheaf Hom, these functors extend their classical counterparts defined at the level of homotopy categories with suitable boundedness conditions. We refer to \cite[Ch. 2]{KS94} for the classical theory and \cite{Jin24,Vol21} for details on its unbounded, $\infty$-categorical extension. We write $\pt_X: X \to \pt$ for the projection, and have the constant and dualizing sheaves $\bC_X := \pt_X^* \bC$ and $\omega_X := \pt_X^! \bC$. 

By default manifold will mean real analytic manifold, as in \cite[Ch. 8]{KS94}. A sheaf on a manifold $M$ is constructible if it has perfect stalks and locally constant restrictions to the strata of a locally finite subanalytic stratification \cite[Def. 8.4.3, Thm. 8.3.20]{KS94}. More generally, a sheaf on a locally closed subanalytic subset $X \subset M$ is constructible if it satisfies this condition. We write $\Sh^b(X) \subset \Sh(X)$ for the subcategory of constructible sheaves. It is closed under the Verdier dual $\bD_X:= \cHom(-,\omega_X)$, which we write as $\bD$ when $X$ is clear from context. If $i: X \into M$ is the inclusion, we sometimes write $\bC_X$ and $\omega_X$ for $i_! \bC_X$ and $i_! \omega_X$ when $i_!$ is clear from context. Note that if $X$ is closed then $i_* \cong i_!$ and $\omega_X \cong \bD_M \bC_X$. 

Associated to $\cF \in \Sh(M)$ is its singular support $ss(\cF) \subset T^* M$, a (typically singular) conic Lagrangian if $\cF$ is constructible. Writing $S^* M$ for the cosphere bundle of $M$, we define the asymptotic boundary $\partial_\infty L \subset S^* M$ of a subset $L \subset T^* M$ as follows. We form the fiberwise spherical compactification of $T^* M$, take the closure of $L$, and then intersect this with the fiberwise boundary. The asymptotic boundary of $ss(\cF)$ will also be denoted by $ss^\infty(\cF)$. 

\subsection{Wrapping functors}
Given a manifold $M$ we write $\inclarg{M}$ for the inclusion $\Loc(M) \subset \Sh(M)$ of locally constant sheaves. It has left and right adjoints denoted by $W^+_M$ and $W^-_M$, respectively. We abbreviate $W^+_M$ to $W_M$ by default, as we use it more often than $W^-_M$. We write $\incl$, $W$, and $W^-$ when $M$ is clear from context. We refer to $W_M$ and $\wrneg_M$ as the positive and negative wrapping functors, as they are respectively a colimit over positive contact isotopies and a limit over negative contact isotopies of $S^* M$ \cite[Thm. 1.2]{Kuo23}. 

We will make use of the following general results. For the first, recall that the convolution with respect to $K \in \Sh(M \times M)$ is the functor $ K \circ - : \Sh(M) \to \Sh(M)$ given by $p_{1!}(K \otimes p_2^*(-)),$ 
where $p_1$ and $p_2$ are the projections to the left and right factors. 

\begin{Lemma}[{\cite[Lemma 3.5.7]{Kuo22}}]\label{lem:W_as_kernel}
Let $M$ be a manifold and $\bC_\Delta \in \Sh(M \times M)$ the constant sheaf on the diagonal. Then $\wr_M$ is isomorphic to $(\wr_{M \times M} \bC_{\Delta}) \circ -$. 
\end{Lemma}

For the next statement, note that $f^*$ preserves local systems for any map $f: M \to N$. Thus $f^* \inclarg{N} \cong \inclarg{M} f^*$, and by adjunction we obtain a map $\wr_M f^* \to f^* \wr_N$. 

\begin{Lemma}\label{lem:W_and_upper*}
Let $f: M \to N$ be a smooth locally trivial fiber bundle. Then the canonical map $\wr_M f^* \to f^* \wr_N$ is an isomorphism. 
\end{Lemma}
\begin{proof}
The right adjoint of the restriction $f^*: \Loc(N) \to \Loc(M)$ is $\wrneg_N f_* \inclarg{M}$. Passing to right adjoints, the claim becomes equivalent to the map $\inclarg{N} \wrneg_N f_* \inclarg{M} \to f_* \inclarg{M}$ being an isomorphism. But this is just the condition that $f_*$ preserves local systems, which holds since $f$ is a smooth locally trivial fiber bundle. 
\end{proof}

\begin{Lemma}\label{lem: compatibility-pushforward-local-system}
Let $f: M \rightarrow N$ be a smooth locally trivial fiber bundle. Then there is a canonical isomorphism $\wr_N f_! \cong f_! \wr_M$.
\end{Lemma}
\begin{proof}
Since $f$ is a submersion $f^! \cong f^* \otimes \omega_f$, hence $f^!$ preserves local systems. Thus we have an isomorphism $f^! \inclarg{N} \cong \inclarg{M} f^!_{Loc}$, where $f^!_{Loc}: \Loc(N) \to \Loc(M)$ denotes the restriction of~$f^!$. Since $\wr_N f_!$ is left adjoint to~$f^!_{Loc}$ we obtain an isomorphism $\wr_N f_! \cong \wr_N f_! \wr_M$ by taking left adjoints. The desired identity follows since $f$ being a smooth locally trivial fiber bundle implies $f_!$ preserves local systems, hence $\wr_N f_! \wr_M \cong f_! \wr_M$. 
\end{proof}

\subsection{Sheaves on tori}\label{sec:torisheaves}
Next we collect some general results about sheaves on $T^n$ and their mirrors. By the mirror of a sheaf $\cF \in \Sh(T^n)$ we will mean its image under 
$$ \Hom(\pi_! \omega_{\bR^n},\wr(-)): \Sh(T^n) \to \Mod_{R_n}, $$
where again we write $R_n$ for $\bC[z_1^{\pm1}, \dotsc, z_n^{\pm1}]$ and $\pi: \bR^n \to T^n$ for the projection. This functor is the composition of $\wr$ with a certain normalization of the monodromy equivalence between local systems on $T^n$ and $R_n$-modules, but it will be convenient to realize the latter equivalence as the functor corepresented by $\pi_! \omega_{\bR^n}$. This in turn requires a standard isomorphism $\End(\pi_! \omega_{\bR^n}) \cong R_n$, which we review here since the details will be needed in Section \ref{sec:sheaves}. Below we write $\tau_m$ for the translation $x \mapsto x + m$ associated to $m \in \bZ^n$. 

\begin{Proposition}\label{prop:endcomp}
	There is a canonical isomorphism $\pi^! \pi_! \cong \bigoplus_{m \in \bZ^n} \tau_{m!}$. In particular, 
	$$\End(\pi_! \omega_{\bR^n}) \cong \bigoplus_{m \in \bZ^n} \Hom(\omega_{\bR^n}, \tau_{m!} \omega_{\bR^n}) \cong R_n,$$
	where for any $m \in \bZ^n$ the second map takes the canonical isomorphism $\omega_{\bR^n} \cong \tau_{m!}\omega_{\bR^n}$ to $z^m$. 
\end{Proposition}

\begin{proof}
	For any $m \in \bZ^n$ we have a diagram 
	\begin{equation*}
		\begin{tikzpicture}
			[baseline=(current  bounding  box.center),thick,>=\arrtip]
			\node (aa) at (0,0) {$\bR^n $};
			\node (ab) at (3.5,0) {$\bR^n \times \bZ^n$};
			\node (ac) at (7,0) {$\bR^n$};
			\node (bb) at (3.5,-1.5) {$\bR^n$};
			\node (bc) at (7,-1.5) {$T^n$.};
			\draw[->] (aa) to node[above] {$i_m $} (ab);
			\draw[->] (ab) to node[above] {$p $} (ac);
			\draw[->] (bb) to node[above] {$\pi $} (bc);
			\draw[->] (aa) to node[below left] {$\tau_m $}(bb);
			\draw[->] (ab) to node[right] {$a $} (bb);
			\draw[->] (ac) to node[right] {$\pi $} (bc);
		\end{tikzpicture}
	\end{equation*}
	where $p$ projects to the first factor, $a$ adds the factors, $i_m$ is the inclusion of $\bR^n \times \{m\}$, and the right square is Cartesian. The sum $\bigoplus_{m \in \bZ^n} i_{m!} i_m^! \to \id$ of counit maps is an isomorphism since it is so locally. The base change map $a_! p^! \to \pi^! \pi_!$ is an isomorphism since $\pi$ is a local homeomorphism (and thus $\pi^! = \pi^*$). Since $\tau_m = a \circ i_m$,  $\id = p \circ i_m$, and since $a_!$ preserves colimits, we thus obtain an isomorphism
	$$ \bigoplus_{m \in \bZ^n} \tau_{m!} \congto \bigoplus_{m \in \bZ^n} a_! i_{m!} i_m^! p^! \congto a_! \bigoplus_{m \in \bZ^n}  i_{m!} i_m^! p^! \congto a_! p^! \congto \pi^! \pi_!. $$
	Now note that the map $\bigoplus_{m \in \bZ^n} \Hom(\omega_{\bR^n}, \tau_{m!} \omega_{\bR^n}) \to \Hom(\omega_{\bR^n}, \bigoplus_{m \in \bZ^n} \tau_{m!} \omega_{\bR^n})$
	is an isomorphism since it may be computed in $\Loc(\bR^n)$, where $\omega_{\bR^n}$ is compact. The stated identification of $\End(\pi_! \omega_{\bR^n})$ now follows by adjunction, since $\Hom(\omega_{\bR^n}, \tau_{m!} \omega_{\bR^n}) \cong \bC$ for all $m$. An elaboration of this argument, which we omit, further identifies $\pi^! \pi_!$ and $\bigoplus_{m \in \bZ^n} \tau_{m!}$ as monads, hence $\End(\pi_! \omega_{\bR^n})$ and $R_n$ as algebras. 
\end{proof}

The choice of $\pi_! \omega_{\bR^n}$ (as opposed to, say, $\pi_! \bC_{\bR^n}$) is explained by the following fact. 

\begin{Proposition}[{\cite[Cor. 3.13]{FLTZ11},\cite[Prop. 9.4]{Kuw20}}]\label{prop:monoidalmirror}
The equivalence $\Hom(\pi_! \omega_{\bR^n}, -): \Loc(T^n) \congto \Mod_{R_n}$ is symmetric monoidal with respect to convolution on $T^n$ and the tensor product of $R_n$-modules. In particular, $\pi_! \omega_{\bR^n}$ is the monoidal unit in $\Loc(T^n)$. 
\end{Proposition}

We will also need the following results relating convolution and wrapping on~$T^n$. Here the convolution product on $\Sh(T^n)$ is defined by
$$ - \conv - := m_!(- \boxtimes -), $$
where $m: T^n \times T^n \to T^n$ is the multiplication map. 

\begin{Proposition}\label{prop:W_as_convolution}
There is a canonical isomorphism $(\wr \bC_{\{\pi(0)\}}) \conv (-) \cong \wr_{T^n}$. 
\end{Proposition}
\begin{proof}
Write $d: T^n \times T^n \to T^n$ for the map $(\theta_1, \theta_2) \mapsto \theta_1 - \theta_2$. Since $m = p_1 \circ (m, p_2)$ and since $(d, p_2) = (m, p_2)^{\shortminus 1}$, we have
\begin{align*}
\wr \bC_{\{\pi(0)\}} \conv (-) &\cong p_{1!} (m, p_2)_! (p_1^*\wr \bC_{\{\pi(0)\}} \otimes p_2^*(-))\\ 
&\cong p_{1!} (d, p_2)^* (p_1^*\wr \bC_{\{\pi(0)\}} \otimes p_2^*(-))\\ 
&\cong p_{1!} (d^*\wr \bC_{\{\pi(0)\}} \otimes p_2^*(-)).
\end{align*}
Noting that $d^* \bC_{\{\pi(0)\}}$ is the constant sheaf on the diagonal, it thus suffices by Lemma \ref{lem:W_as_kernel} to show that the canonical map $\wr d^* \bC_{\{\pi(0)\}} \to d^*\wr \bC_{\{\pi(0)\}}$ is an isomorphism. But this follows from Lemma \ref{lem:W_and_upper*}.
\end{proof}

\begin{Proposition}\label{prop:monoidalwrapping}
The restriction of the convolution product on $\Sh(T^n)$ to $\Loc(T^n)$ extends to a symmetric monoidal structure with unit $\wr \bC_{\{\pi(0)\}}$, and for which the functor $\wr_{T^n}$ is symmetric monoidal. 
\end{Proposition}
\begin{proof}
Suppose that $\cF, \cG, \cH\in \Sh(T^n)$, and that $\phi: \cF \to \cG$ is such that $\wr\phi: \wr \cF \to \wr \cG$ is an isomorphism. By Proposition \ref{prop:W_as_convolution} we can identify $\wr\phi$ with $\id_{\wr \bC_{\{\pi(0)\}}} \conv \phi$ and $\wr(\phi \conv \id_{\cH})$ with $ \id_{\wr \bC_{\{\pi(0)\}}} \conv \phi \conv  \id_{\cH}$, hence $\wr(\phi \conv \id_{\cH})$ is also an isomorphism. It follows from \cite[Ex. 2.2.1.7, Prop. 2.2.1.9]{LurHA} that $\Loc(T^n)$ carries a unique symmetric monoidal structure such that $\wr_{T^n}$ is symmetric monoidal and $\inclarg{T^n}$ is lax symmetric monoidal. But $\Loc(T^n)$ is closed under convolution in $\Sh(T^n)$, hence the product for this symmetric monoidal structure is the restriction of the convolution product and its unit is $\wr \bC_{\{\pi(0)\}}$. 
\end{proof}

Note that the isomorphism $\wr_{T^n}(- \conv -) \cong \wr_{T^n} (-) \conv \wr_{T^n} (-)$ follows immediately from Proposition \ref{prop:W_as_convolution} and the elementary calculation that $\wr \bC_{\{\pi(0)\}} \cong \wr \bC_{\{\pi(0)\}} \conv \wr \bC_{\{\pi(0)\}}$. In particular, the role of \cite[Prop. 2.2.1.9]{LurHA} is merely to provide the needed coherence data. Also note that Propositions \ref{prop:monoidalmirror} and \ref{prop:monoidalwrapping} together imply $\wr \bC_{\{\pi(0)\}} \cong \pi_! \omega_{\bR^n}$; this will also follow from the more general calculation of Lemma \ref{lem: wrapping-of-contractbles1}. 

Finally, we will need the essentially standard fact that $\Sh^b(T^n)$ is rigid as a symmetric monoidal category, i.e. every object is dualizable with respect to the convolution product. We include a proof for convenience --- note that, by contrast, a constructible sheaf on a noncompact group such as $\bR^n$ is not necessarily dualizable \cite[Lem. 7.4]{FLTZ11}. Below we write $\cF \mapsto \shortminus \cF$ for the autoequivalence of $\Sh^b(T^n)$ given by pullback along $\theta \mapsto \shortminus \theta$.

\begin{Proposition}\label{prop:rigidity}
The symmetric monoidal category $\Sh^b(T^n)$ is rigid, and the dual of $\cF \in \Sh^b(T^n)$ is $\shortminus \bD \cF$. 
\end{Proposition}
\begin{proof}
Given $\cF \in \Sh^b(T^n)$, write $\lambda_\cF: \Sh^b(T^n) \to \Sh^b(T^n)$ for $\cF \conv - \cong m_!(\cF \boxtimes -)$. It has a right adjoint $\lambda_\cF^R \cong p_{2*} \cHom(p_1^*(\cF), -)m^!$, where $p_1, p_2, m: T^n \times T^n \to T^n$ are the left/right projections and the multiplication map. For any $\cG \in \Sh^b(T^n)$, the isomorphism $(\cF \conv -) \conv \cG \cong \cF \conv (- \conv \cG)$ induces a map $\lambda_\cF^R(-) \conv \cG \to \lambda_\cF^R(- \conv \cG)$. The dualizability of $\cF$ is equivalent to this map being an isomorphism (\cite[Lem. I.1.9.1.5]{GR17}, see also \cite[Lem. 6.2]{CW2}). By construction, it factors as the following composition of natural maps, where e.g. $m_{23}: T^n \times T^n \times T^n \to T^n \times T^n$ denotes multiplication of the second and third factors. 
\begin{align*}
m_! (- \boxtimes \cG) p_{2*} \cHom(p_1^*\cF,-)m^! & \to m_!  p_{23*}(- \boxtimes \cG)  \cHom(p_1^*\cF,-)m^! \\
& \to m_!  p_{23*} \cHom(p_1^*\cF,-)(- \boxtimes \cG) m^! \\
& \to m_!  p_{23*} \cHom(p_1^*\cF,-)m_{12}^! (- \boxtimes \cG)  \\
& \to p_{2*} m_{23!} \cHom(p_1^*\cF,-)m_{12}^! (- \boxtimes \cG)  \\
& \to p_{2*} \cHom(p_1^*\cF,-) m_{23!} m_{12}^! (- \boxtimes \cG)  \\
& \to p_{2*} \cHom(p_1^*\cF,-) m^! m_! (- \boxtimes \cG). 
\end{align*}
The first map is an isomorphism by proper base change and \cite[Prop. 2.6.6]{KS94}, the second is by \cite[Prop. 3.4.4, Prop. 3.4.6]{KS94} (cf. \cite[Cor. 2.9.6]{Ach21}), and the third is by \cite[Prop. 3.3.2]{KS94} since $m$ is a submersion. The fourth map is an isomorphism since $T^n$ is compact, the fifth is by \cite[eq. 2.6.15]{KS94}, and the last is by proper base change. Given that $\cF$ is dualizable, its dual is necessarily $\lambda_\cF^R(\bC_{\{\pi(0)\}})$, which is isomorphic to $\shortminus \bD \cF$ by the argument of \cite[Thm. 7.3]{FLTZ11}. 
\end{proof}

\subsection{Sheaf quantization}\label{sec:sheafquant} Let $M$ be a manifold and $L \subset T^* M$ an embedded exact Lagrangian. If $L$ is equipped with suitable brane data, one can associate to it a constructible sheaf $Q(L)$ on $M$. Such sheaf quantizations of $L$ have been considered by various authors from different points of view \cite{NZ09,GKS12,Gui16,JT17,Vit19,GPS18,Li23}. In Section~\ref{sec:branes} we will consider the sheaf quantizations of certain Lagrangian branes in the formalism of \cite{JT17}, the details of which we briefly recall here. 

We assume that $L$ is eventually conical, i.e. outside a fiberwise compact subset it coincides with an $\bR_+$-invariant Lagrangian. Following \cite[Sec. 3.6]{JT17}, up to contractible choices there is a unique perturbation $L'$ of $L$ which is lower exact, i.e. the restriction of the Liouville form to $L'$ has a primitive $f$ which is proper and bounded above. The perturbation $L'$ lifts to an embedded Legendrian in $S^*(M \times \bR)$, canonical up to translation in the $\bR$-direction. We obtain a category $\mu sh(L')$ of microlocal sheaves or brane structures on $L'$, which is locally (but not globally) the quotient of $\Sh_{L'}(M \times \bR)$ by $\Loc(M \times \bR)$ (\cite[Sec. 3.11]{JT17} following \cite[Sec. 6]{KS94}, see also \cite[Rem. 3.6]{CL24}). The category $\mu sh(L')$ is locally equivalent to $\Loc(L')$, but the two categories are not always globally equivalent. 

Write $\Sh^0_{L'}(M \times \bR) \subset \Sh_{L'}(M \times \bR)$ for the subcategory of sheaves vanishing on $M \times (R, \infty)$ for $R \gg 0$. Also write $\Lambda \subset S^* M$ for the asymptotic boundary of $L$. Following \cite[Sec. 3.18]{JT17}, we have a canonical functor
\begin{equation}\label{eq:sheafquant}
\mu sh(L') \xrightarrow{\sim} \Sh^0_{L'}(M \times \bR) \xrightarrow{p_*} \Sh_{\Lambda}(M),
\end{equation} 
where the left functor is the inverse of microlocalization and the right is projection to $M$. This projection is fully faithful if $\Lambda$ is smooth. A Lagrangian brane supported on $L$ will for us mean $L$ together with the data of an object of $\mu sh(L')$. If suitable obstructions vanish, the latter can be identified with an object of $\Loc(L') \cong \Loc(L)$. Given such data we write $Q(L) \in \Sh_\Lambda(M)$ for the object arising from the composition (\ref{eq:sheafquant}).

We caution that while all notions of sheaf quantization in the above references should agree, we are not aware of this having been documented in general. In particular, in the introduction we mentioned as motivation that $WQ(L) \cong L$ under the equivalence $\Loc^c(M) \cong \sW(T^* M)$. This follows from results of \cite{GPS18} in the framework described there, in which we would define $Q(L)$ as the image of a negative Reeb perturbation of $L$ under their equivalence $\Sh_{\Lambda}^c(M) \cong \sW(T^* M; \Lambda)$. We do not know a comparison result between this quantization and that of \cite{JT17}, but we will also never actually use the identification $WQ(L) \cong L$. 

\subsection{Real nearby cycles}

Given a manifold $M$ and $\cF_{\opint} \in \Sh(M \times \opint)$, we write $\cF_t$ for the restriction of $\cF_{\opint}$ to $M \times \{t\}$. One can make sense of the limit of the $\cF_t$ as $t \to 0$ as the (real) nearby cycles (or specialization) of $\cF_{\opint}$, defined as
$$ \psi \cF_{\opint} := i^* j_{*} \cF_{\opint}.$$
Here $i: M \times \{0 \} \into M \times \clint$ and $j: M \times \opint \into M \times \clint$ are the inclusions.

In Section \ref{sec:branes} we will need the following facts about nearby cycles. Consider the following diagram of Cartesian squares, where $a: Z \into M \times \opint$ is a closed subset and $\ol{Z}$ its closure in $M \times \clint$. 
\begin{equation}\label{eq:Zsquares}
	\begin{tikzpicture}
		[baseline=(current  bounding  box.center),thick,>=\arrtip]
		\node (aa) at (0,0) {$Z_0$};
		\node (ab) at (3.5,0) {$\ol{Z}$};
		\node (ac) at (7,0) {$Z$};
		\node (ba) at (0,-1.5) {$M \times \{0\}$};
		\node (bb) at (3.5,-1.5) {$M \times \clint$};
		\node (bc) at (7,-1.5) {$M \times \opint$.};
		\draw[->] (aa) to node[above] {$i' $} (ab);
		\draw[<-] (ab) to node[above] {$j' $} (ac);
		\draw[->] (ba) to node[above] {$i $} (bb);
		\draw[<-] (bb) to node[above] {$j $} (bc);
		\draw[->] (aa) to node[left] {$a_0 $}(ba);
		\draw[->] (ab) to node[right] {$\ol{a} $} (bb);
		\draw[->] (ac) to node[right] {$a$} (bc);
	\end{tikzpicture}
\end{equation}

\begin{Lemma}\label{lem:vancycles}
	Suppose that in (\ref{eq:Zsquares}) the natural map $\bC_{\ol{Z}} \to j'_* \bC_Z$ is an isomorphism. Then we have $\psi a_* \bC_Z \cong a_{0*} \bC_{Z_0}$. 
\end{Lemma}
\begin{proof}
	Equivalently, the hypothesis states that the natural map $\ol{a}_* \bC_{\ol{Z}} \to j_* a_* \bC_{Z}$ is an isomorphism. But then
	$$ i^* j_* a_* \bC_{Z} \cong i^* \ol{a}_* \bC_{\ol{Z}} \cong a_{0*} \bC_{Z_0},$$
	where the right isomorphism follows by base change, since the vertical maps are closed embeddings. 
\end{proof}

The hypothesis of Lemma \ref{lem:vancycles} should be understood as a tameness condition on the topology of $Z$ near $M \times \{0\}$. A nonexample is provided by $M = \bR$ and $Z = \{(\sin(\frac1x),x)\}_{x \in \opint}$, in which case $j'_* \bC_Z$ has infinite-dimensional stalks along $\ol{Z} \smallsetminus Z$. In particular, $j'_* \bC_Z$ is not constructible even though $\bC_Z$ is, providing a nonexample for the following lemma as well. 

\begin{Lemma}\label{lem:verdiervan}
	Suppose that $\cF \in \Sh(M \times \opint)$ is constructible and so is $j_! \bD \cF$.  Then we have $\psi \bD \cF \cong (\bD \psi \cF) [1]$. 
\end{Lemma}
\begin{proof}
	Noting that $j_! j^! j_* \cF \cong j_! \cF$ since $j$ is open, we have a triangle
	\begin{equation}\label{eq:verdiervan1} j_! \cF \to j_* \cF \to i_* i^* j_* \cF. \end{equation}
	Since $\cF$ and $j_! \bD \cF$ are constructible we have $\bD j_* \cF \cong j_! \bD \cF$ \cite[Ex. VIII.3]{KS94}, 
	 hence applying~$\bD$ to (\ref{eq:verdiervan1}) we obtain 
	\begin{equation}\label{eq:verdiervan2} i_* \bD i^* j_* \cF \to j_! \bD \cF \to j_* \bD \cF. \end{equation}
	But the second map in (\ref{eq:verdiervan2}) is the first map in (\ref{eq:verdiervan1}) with $\bD \cF$ in place of $\cF$, hence the cone over it is $i_* i^* j_* \bD \cF$. Thus $i_* i^* j_* \bD \cF \cong (i_* \bD i^* j_* \cF)[1]$, and applying $i^*$ we obtain the desired isomorphism. 
\end{proof}

\section{Coamoebae from resolutions}\label{sec:restoco}

In this section we associate a simplicial complex $X(F^\bullet)$ over $T^n$ to a bounded complex $F^\bullet$ of free $R_n$-modules (Definition \ref{def:XFbullet}). We show that the discrete information in $F^\bullet$ can be recovered from $X(F^\bullet)$ (Theorem \ref{thm:recovery}), and in common situations from its image $T(F^\bullet) \subset T^n$ (Theorem~\ref{thm:recovery2}). A linearly immersed graph arises from the construction if and only if it is bipartite, and more generally we characterize the $k$-dimensional simplicial complexes arising from this construction as those containing certain ``generating'' subgraphs with vertices colored by~$[\shortminus k, 0] \subset \bZ$ (Proposition \ref{prop:characterization}). 

More precisely, we fix throughout a bounded sequence
$$ F^\bullet = \cdots \to F^k \xrightarrow{d^k} F^{k+1} \to \cdots $$
of based finite-rank free $R_n$-modules and homomorphisms (recall that $R_n := \bC[z_1^{\pm 1},  \dotsc, z_n^{\pm 1}]$). The given modules being based means that for each $k$ we have fixed a finite set $I_k$ and an isomorphism $F^k \cong R_n^{\,\oplus I_k}$. We write $I$ for the disjoint union of the $I_k$, and define $\deg: I \to \bZ$ so that $\deg(i) = k$ if $i \in I_k$. The term sequence indicates that we do not actually require $d^2 = 0$, since the constructions in this section will not depend on this. For convenience, however, we do assume that $F^0 \ncong 0$ and that $F^k \cong 0$ for $k > 0$. 

Each $d^k$ may be represented as an $R_n$-valued $I_{k+1} \times I_k$ matrix. We write $d^k_{ij} \in R_n$ for the entry in row $i \in I_{k+1}$ and column $j \in I_k$. To describe our main construction, it will be convenient to package the exponents appearing in these entries as follows. 

\begin{Definition}\label{def:Mij}
	For each $i,j \in I$ we define $\Exp_{ij} \subset \bZ^n$ as follows.
	\begin{itemize}
		\item If $\deg(i) \leq \deg(j)$, then $\Exp_{ij}$ is empty.
		\item If $\deg(i) = \deg(j) + 1$, then $\Exp_{ij}$ is determined by the condition that $d^{\,\deg(j)}_{ij} = \sum_{m \in \Exp_{ij}} c_{ijm} z^m$ with $c_{ijm} \neq 0$ for all $m \in \Exp_{ij}$. 
		\item If $\deg(i) = \deg(j) + k$ for $k > 1$, then $m \in \Exp_{ij}$ if $m = \sum_{\ell = 1}^{k} m_\ell$ for some $i_1, \dotsc, i_{k-1} \in I$ and some $(m_1, \dotsc, m_{k}) \in \Exp_{i_1 j} \times \Exp_{i_2 i_1} \times \cdots \times \Exp_{i i_{k-1}}$ such that $\deg(i_\ell) = \deg(j) + \ell$ for $1 \leq \ell < k$. 
	\end{itemize}
\end{Definition}

Informally, if $\deg{i} > \deg{j} + 1$ then $\Exp_{ij}$ is the set of all exponents that appear when computing the matrix entry $(d^{\,\deg(i)-1} \cdots d^{\,\deg(j)})_{ij}$, even if all monomials with such an exponent are canceled out in the final result (as happens in particular when $F^\bullet$ is a cochain complex). 

A semi-simplicial complex over $T^n$ is a collection $X = \{X_k\}_{k \in \bN}$, where each $X_k$ is a set of maps $\sigma: \Delta^k \to T^n$ whose restrictions to each $m$-dimensional face belong to $X_m$. We say $X$ is immersed if for all $k \in \bN$ and $\sigma \in X_k$ the restriction of $\sigma$ to the interior of $\Delta^k$ is an embedding. We say $X$ is linear if each $\sigma$ is linear, i.e. it is the restriction of $\pi: \bR^n \to T^n$ to the convex hull of an ordered tuple of points in $\bR^n$. Given such a tuple $y_0, \dotsc, y_k \in \bR^n$, we write $[y_0, \dotsc, y_k]$ for its convex hull and $[y_0, \dotsc, y_k] \!\!\mod \bZ^n$ for the associated simplex over~$T^n$. 

\begin{Definition} \label{def:XFbullet}
	We write $X(F^\bullet) = \{X(F^\bullet)_k\}_{k \in \bN}$ for the following collection of simplices over $T^n$. First we choose a point $x_i \in \bR^n$ for all $i \in I$. Then we set 
	\begin{equation*} 
		X(F^\bullet)_k := \bigcup_{i_0, \dotsc, i_k \in I}\{[x_{i_0}, x_{i_1} + m_1, \dotsc, x_{i_k} + \sum_{\ell = 1}^{k} m_\ell] \!\!\!\!\mod \bZ^n\}_{ (m_1,\dotsc,m_k) \in \Exp_{i_1 i_0} \times \cdots \times \Exp_{i_k i_{k\shortminus 1}}},
	\end{equation*}
	the $k = 0$ case being interpreted as
	\begin{equation*}
		X(F^\bullet)_0 := \bigcup_{i_0 \in I}\{[x_{i_0}] \!\!\!\!\mod \bZ^n\}.
	\end{equation*}
We write  $T(F^\bullet) \subset T^n$ for the image of $X(F^\bullet)$. 
\end{Definition} 

The collection $X(F^\bullet)$ has the following properties. Here and below when the $x_i$ being generic appears as a hypothesis, we mean that the set of choices of $x_i$ where the conclusion does not hold is a measure zero subset of $(\bR^n)^I$. 

\begin{Proposition}\label{prop:immersedsimpset}
	The collection $X(F^\bullet)$ is a linear semi-simplicial set over~$T^n$. If the $x_i$ are generic and $F^{\shortminus k} \cong 0$ for $k \geq n$, then $X(F^\bullet)$ is immersed. 
\end{Proposition}
\begin{proof}
	Linearity is by construction. To see closure under restriction to faces, it suffices by induction to consider facets. Consider the $k$-simplex associated to $i_0, \dotsc, i_k \in I$ and $(m_1,\dotsc,m_k) \in \Exp_{i_1 i_0} \times \cdots \times \Exp_{i_k i_{k\shortminus 1}}$, and choose some $0 \leq j \leq k$. We respectively define $i'_\ell$ for $0 \leq \ell <k$ and $m'_\ell$ for $1 \leq \ell < k$ by 
	$$ i'_\ell = \begin{cases}
		i_\ell & \ell < j \\
		i_{\ell + 1} & \ell \geq j,
	\end{cases}
	\quad
	m'_\ell = \begin{cases}
		m_\ell & \ell < j \\
		m_j + m_{j+1} & \ell = j \\
		m_{\ell+1} & \ell > j.
	\end{cases} $$ 
	It follows from the definition of $\Exp_{i'_j i'_{j\shortminus 1}} = \Exp_{i_{j + 1}i_{j\shortminus 1}}$ that $(m'_1, \dotsc, m'_{k\shortminus 1})$ belongs to $\Exp_{i'_1 i'_0} \times \cdots \times \Exp_{i'_{k\shortminus 1} i'_{k\shortminus 2}}$, and it follows from inspection that the $(k\shortminus 1)$-simplex over $T^n$ associated to $(i'_0, \dotsc, i'_{k\shortminus 1})$ and $(m'_1, \dotsc, m'_{k\shortminus 1})$ is the $j$th facet of the given simplex. Note in checking the $j=0$ case that we have an equality 
	$$[x_{i_1}, \dotsc, x_{i_k} + \sum_{\ell = 2}^k m_\ell] \!\!\!\!\mod \bZ^n = [x_{i_1} + m_1, \dotsc, x_{i_k} + \sum_{\ell = 1}^k m_\ell] \!\!\!\!\mod \bZ^n$$ of $(k \shortminus 1)$-simplices over $T^n$, even though the two expressions involve distinct simplices in $\bR^n$. 
	
	Now suppose that $F^{\shortminus k} \cong 0$, hence that $X(F^\bullet)_k$ is empty, for $k \geq n$. We want to show the condition that $\sigma_k$ is not an embedding on $\Int\,\Delta^k$ for some $k < n$ and some $\sigma_k \in X(F^\bullet)_k$ is a measure zero condition on $(x_i)_{i \in I} \in (\bR^n)^I$. Since there are finitely many simplices, it suffices to show this is a measure zero condition for any fixed $\sigma_k$. We have $\sigma_k = [x_{i_0}, x_{i_1} + \ol{m}_1, \dotsc, x_{i_k} + \ol{m}_k] \!\!\!\mod \bZ^n$ for some $i_0, \dotsc, i_k \in I$ and $\ol{m}_1, \dotsc, \ol{m}_k \in \bZ^n$.  Set $\widehat{\sigma}_k := [x_{i_0}, x_{i_1} + \ol{m}_1, \dotsc, x_{i_k} + \ol{m}_k] \subset \bR^n$. By genericity of the $x_i$ we may assume the vertices of $\widehat{\sigma}_k$ are affinely independent, and then write $\Int\,\widehat{\sigma}_k$ for the preimage of $\Int\,\Delta^k$ under $\widehat{\sigma}_k \cong \Delta^k$ (strictly speaking, $\widehat{\sigma}_k$ has empty interior as a subset of $\bR^n$).  
	The condition that $\sigma_k$ is not an embedding on $\Int\,\Delta^k$ is the condition that $(\Int\, \widehat{\sigma}_k) \cap (\Int\, \widehat{\sigma}_k + m)$ is nonempty for some nonzero $m \in \bZ^n$. Since there are countably many $m$, it suffices to show this is a measure zero condition for any fixed~$m$. 
	
	For this it suffices to show the condition that the $k$-planes spanned by $\Int\, \widehat{\sigma}_k$ and $\Int\, \widehat{\sigma}_k + m$ have nonempty intersection is a measure zero condition. These $k$-planes are parallel, so their intersection is nonempty if and only if they coincide. This happens if and only if $m$ is in the linear span of the $x_{i_j} + \ol{m}_j - x_{i_0}$ for $1 \leq j \leq k$, which is a measure zero condition since $k < n$. 
\end{proof}

Note that $X(F^\bullet)$ is equivalent to the data of the simplicial complex over $T^n$ obtained by adjoining the evident degenerate simplices. Our exposition sometimes conflates the two, but the semi-simplicial complex above is more convenient in practice. 

Our next goal is to make precise what information in $F^\bullet$ is encoded by $X(F^\bullet)$. As already discussed, we summarize the answer as being the discrete information in $F^\bullet$. To formalize this, note that the sets $\Exp_{ij}$ can be packaged as a function $\Exp: I \times I \to P(\bZ^n)$, where $P(\bZ^n)$ is the power set of $\bZ^n$ (this package is of course highly redundant, as $\Exp$ is determined by its values on pairs $i,j$ with $\deg(i) = \deg(j) + 1$). 

\begin{Definition}
	The discrete information in $F^\bullet$ consists of  the index set $I$, the degree function $\deg: I \to \bZ$, and the function $\Exp: I \times I \to P(\bZ^n)$. 
\end{Definition}

We regard two collections of such information as equivalent if they differ by the following superficial operations. First, we can transfer the rest of the data across any bijection $I \cong I'$ of index sets. Second, for any $i \in I$ and $m \in \bZ^n$ we can add $m$ to each element of $\Exp_{ij}$ for all $j$. We can also add $\shortminus m$ to each element of $\Exp_{ji}$ for all $j$. This corresponds to rescaling the $i$th summand of~$F^\bullet$ by $z^m$, which does not affect anything of interest. 

To justify our terminology, note that the above information determines $F^\bullet$ up to a finite-dimensional continuous ambiguity. That is, consider the family of sequences
$$ \cdots \to R_n^{\,\oplus I_k} \xrightarrow{d^k} R_n^{\,\oplus I_{k+1}} \to \cdots, $$
where $d^k$ ranges over all $I_{k+1} \times I_k$ matrices whose $ij$-entry is of the form $\sum_{m \in \Exp_{ij}} c_m z^m$ with each $c_m$ nonzero. This family contains $F^\bullet$, and is parametrized by an algebraic torus whose dimension is the sum of the cardinalities $|\Exp_{ij}|$ as $i, j$ range over all pairs with $\deg(i) = \deg(j) + 1$. 

\begin{Theorem}\label{thm:recovery}
	If the $x_i$ are distinct modulo $\bZ^n$, then for any $i, j \in I$ we have
	\begin{equation}\label{eq:Mijrecovery}
		\Exp_{ij} = \{m \in \bZ^n \,|\, [x_j, x_i + m] \!\!\!\!\mod \bZ^n \in X(F^\bullet)_1 \}. 
	\end{equation}
	In particular, the discrete information in $F^\bullet$ can be recovered up to equivalence from $X(F^\bullet)$ and the function $\deg: X(F^\bullet)_0 \to \bZ$. 
\end{Theorem}
\begin{proof}
	By definition $\Exp_{ij}$ is contained in the given subset of $\bZ^n$. On the other hand, if $[x_j, x_i + m] \!\!\!\mod \bZ^n \in X(F^\bullet)_1$ then $[x_j, x_i + m] \!\!\!\mod \bZ^n = [x_\ell, x_k + m'] \!\!\!\mod \bZ^n$ for some $k, \ell \in I$ and $m' \in \bZ^n$. If the $x_i$ are distinct modulo $\bZ^n$ then we must have $i = k$ and $j = \ell$, and the needed containment follows since the 1-simplices $[x_j, x_i + m] \!\!\!\mod \bZ^n$ and $[x_j, x_i + m'] \!\!\!\mod \bZ^n$ are distinct for $m \neq m'$. 
	
	For the final recovery claim note that $I \cong X(F^\bullet)_0$, and that by hypothesis we then know the function $\deg: I \to \bZ$. To recover $\Exp$ from (\ref{eq:Mijrecovery}) we take each $x_i$ to be an arbitrary preimage in $\bR^n$ of the associated element of $X(F^\bullet)_0$. Any two preimages differ by some $m \in \bZ^n$, and the resulting functions $\Exp$ will differ by adding $m$ to each element of $\Exp_{ij}$ and $\shortminus m$ to each element of $\Exp_{ji}$ for all $j$. Thus up to equivalence $\Exp$ is independent of the choice of preimage. 
\end{proof}

In general, $F^\bullet$ and $X(F^\bullet)$ may be far from determined by $T(F^\bullet)$. For example, the latter can be all of $T^n$, since we have placed no bounds on the length of $F^\bullet$ and hence no bounds on the dimension of $X(F^\bullet)$. But when $T(F^\bullet)$ is of positive codimension the situation is better, and we have the following recovery theorem. Here we write $T(F^\bullet)_0 \subset T^n$ for the image of $X(F^\bullet)_0$, noting that when the $x_i$ are distinct modulo $\bZ^n$ we have $T(F^\bullet)_0 \cong I$. We emphasize that this case includes many of interest: an $R_n$-module $M$ has projective dimension at most~$n$, so except when this upper bound is achieved, $T(F^\bullet)$ will be of positive codimension for any minimal free resolution of $M$. 

\begin{Theorem}\label{thm:recovery2}
	Suppose the $x_i$ are generic and the interior of $T(F^\bullet)$ is empty. Then for $i_0, \dotsc, i_k \in I$ and $\ol{m}_1, \dotsc, \ol{m}_k \in \bZ^n$ the simplex $\sigma_k = [x_{i_0}, x_{i_1} + \ol{m}_1, \dotsc, x_{i_k} + \ol{m}_k] \!\!\!\mod \bZ^n$ belongs to $X(F^\bullet)_k$ if and only if $\deg(i_0) < \cdots < \deg(i_k)$ and its image is contained in $T(F^\bullet)$. 
	In particular, the discrete information in $F^\bullet$ can be recovered up to equivalence from $T(F^\bullet)$, $T(F^\bullet)_0$, and the function $\deg: T(F^\bullet)_0 \to \bZ$. 
\end{Theorem}

\begin{Lemma}\label{lem:genericproperty}
	Suppose the $x_i$ are generic, and let $\sigma_k = [x_{i_0}, x_{i_1} + \ol{m}_1, \dotsc, x_{i_k} + \ol{m}_k] \!\!\!\mod \bZ^n$ for some $k < n$, some $i_0, \dotsc, i_k \in I$ with $\deg(i_0) < \cdots < \deg(i_k)$ and some $\ol{m}_1, \dotsc, \ol{m}_k \in \bZ^n$. Then there exists $\theta \in \sigma_k(\Int\, \Delta^k)$ with the following property: if $\sigma_\ell = [x_{j_0}, x_{j_1} + \ol{m}'_1, \dotsc, x_{j_\ell} + \ol{m}'_\ell] \!\!\!\mod \bZ^n$ for some $\ell < n$, some $j_0, \dotsc, j_\ell \in I$ with $\deg(j_0) < \cdots < \deg(j_\ell)$, and some $\ol{m}'_1, \dotsc, \ol{m}'_\ell \in \bZ^n$, then $\theta \notin \sigma_\ell(\Delta^\ell)$ unless $\sigma_k$ is a face of $\sigma_\ell$. 
\end{Lemma}

\begin{proof}
	As in the proof of Proposition \ref{prop:immersedsimpset} we set $\widehat{\sigma}_k := [x_{i_0}, x_{i_1} + \ol{m}_1, \dotsc, x_{i_k} + \ol{m}_k] \subset \bR^n$ and write $\Int\,\widehat{\sigma}_k$ for the preimage of $\Int\,\Delta^k$ under $\widehat{\sigma}_k \cong \Delta^k$ (assuming by genericity of the $x_i$ that the vertices of $\widehat{\sigma}_k$ are affinely independent).  We first claim that for any $j \in I$ and $m \in \bZ^n$, the point $x_j + m$ is contained in the $k$-plane through $\widehat{\sigma}_k$ if and only if $j = i_{a}$ and $m = \ol{m}_a$ for some $0 \leq a \leq k$ (where $\ol{m}_0=0$). To see this, note that for any choice of $j$ and $m$ not of this form, the condition that $x_j + m$ belongs to this $k$-plane defines a positive codimension, hence measure zero, subset of $(\bR^n)^I$ (because $k < n$). There are countably many choices of $j$ and~$m$ not of this form, hence the union of all such subsets is also measure zero. 
	
	Now note that the desired claim is equivalent to 
	\begin{equation}\label{eq:intersections}
		\bigcup_{\ell < n}\left(\bigcup_{\substack{\text{$\sigma_\ell$ w/o $\sigma_k$}\\ \text{ as a face}} }\sigma_k(\Int\, \Delta^k) \cap \sigma_\ell(\Delta^\ell)\right)
	\end{equation}
	being a proper subset of $\sigma_k(\Int\, \Delta^k)$, where the inner union is over all $\sigma_\ell$ of the form given in the statement which do not have $\sigma_k$ as a face. By the proof of Proposition \ref{prop:immersedsimpset} we already know that for generic $x_i$ the map $\pi$ restricts to a homeomorphism between $\Int\,\widehat{\sigma}_k$ and $\sigma_k(\Int\, \Delta^k)$.  Thus it suffices to show the preimage of (\ref{eq:intersections}) in $\Int\,\widehat{\sigma}_k$ is measure zero for the standard $k$-dimensional measure. The union in (\ref{eq:intersections}) is countable, so it suffices to show the preimage of each $\sigma_k(\Int\, \Delta^k) \cap \sigma_\ell(\Delta^\ell)$ is measure zero. 
	
	For any specific $\sigma_\ell$ we can write this preimage as 
	$ \bigcup_{m \in \bZ^n} (\Int\,\widehat{\sigma}_k) \cap (\widehat{\sigma}_\ell + m),$ defining $\widehat{\sigma}_\ell$ the same way as  $\widehat{\sigma}_k$. Thus it suffices to show $(\Int\,\widehat{\sigma}_k) \cap (\widehat{\sigma}_\ell + m)$ is of measure zero in $\Int\,\widehat{\sigma}_k$ for any $m \in \bZ^n$. Because $\sigma_k$ is not a face of $\sigma_\ell$ we have that for some $0 \leq a \leq k$ the vertex $x_{i_a} + \ol{m}_a$ of $\widehat{\sigma}_k$ is not a vertex of $\widehat{\sigma}_\ell + m$. We then have that $x_{i_a} + \ol{m}_a$ does not lie in the $\ell$-plane through $\widehat{\sigma}_\ell + m$ by the first paragraph (now applied to $\widehat{\sigma}_\ell + m$ instead of $\widehat{\sigma}_k$). Since $k,\ell < n$ this $\ell$-plane thus meets the $k$-plane through $\widehat{\sigma}_k$ in an affine subspace of positive codimension, hence this subspace and its further intersection with $\Int\,\widehat{\sigma}_k$ are of measure zero.  
\end{proof}

\begin{proof}[Proof of Theorem \ref{thm:recovery2}]
	By definition every $k$-simplex of $X(F^\bullet)$ is of the given form for some $i_0, \dotsc, i_k$ and $\ol{m}_1, \dotsc, \ol{m}_k$. Suppose a $k$-simplex $\sigma_k$ of this form has image contained in~$T(F^\bullet)$. Since the $x_i$ are generic and the interior of $T(F^\bullet)$ is empty we must have $k < n$. We may thus choose a point $\theta \in \sigma_k(\Int\,\Delta^k)$ satisfying the conclusion of Lemma \ref{lem:genericproperty}. Since $\theta \in T(F^\bullet)$ we must have $\theta \in \sigma_\ell(\Delta^\ell)$ for some $\ell$ and some $\sigma_\ell \in X(F^\bullet)_\ell$. Again since the $x_i$ are generic and the interior of $T(F^\bullet)$ is empty we must have $\ell < n$. But by definition any $\ell$-simplex in $X(F^\bullet)_\ell$ is of the form stated in Lemma  \ref{lem:genericproperty}, hence by its conclusion $\sigma_k$ is a face of $\sigma_\ell$ and is thus also a simplex of $X(F^\bullet)$. Thus we can recover $X(F^\bullet)$ and $X(F^\bullet)_0$ from $T(F^\bullet)$ and $T(F^\bullet)_0$, and the final recovery claim follows from Theorem \ref{thm:recovery}. 
\end{proof}

Next we consider which linear simplicial complexes over $T^n$ arise from Definition~\ref{def:XFbullet}. The question is usefully framed by the case when $F^\bullet$ has only two terms. This implies that $T(F^\bullet)$ is a graph with linear edges (possibly crossing each other if $n=2$). The degree function $T(F^\bullet)_0 \to \{0,\shortminus 1\}$ may be regarded as a bipartite coloring, and conversely a graph in $T^n$ with linear edges is of this form if and only if it is bipartite. 

This recognition criterion generalizes as follows. First note that if $F^\bullet$ has $k$ terms then $X(F^\bullet)$ only has simplices of dimension less than $k$, and the degree function defines a $k$-coloring of the 1-skeleton of $X(F^\bullet)$. It follows from the constructions that $X(F^\bullet)$ is determined by this $k$-coloring together with the subgraph of its 1-skeleton formed by its degree-one edges. Here and below given a simplicial complex $X$ and an integer-valued vertex coloring $\deg: X_0 \to \bZ$, we say the degree of an edge $\sigma: [0,1] \to T^n$ belonging to $X_1$ is $\deg(\sigma(1)) - \deg(\sigma(0))$. 

The existence of a vertex coloring such that $X$ is determined by its degree-one edges in this sense turns out to characterize when $X$ is of the form $X(F^\bullet)$. More precisely, the first condition in the following statement characterizes how the degree-one edges of $X$ should determine its remaining edges, and the second how these should determine the higher-dimensional simplices. 

\begin{Proposition}\label{prop:characterization}
	Let $X$ be a finite, linear semi-simplicial complex over $T^n$ whose vertices have pairwise distinct images. Then $X \cong X(F^\bullet)$ for some bounded sequence $F^\bullet$ of finite-rank free modules if and only if there exists a function $\deg: X_0 \to \bZ$ such that the following hold.
		\begin{enumerate}
		\item An edge $[x,y]\!\!\! \mod \bZ^n$ belongs to $X_1$ if and only if for some $k$ there exist $x_0 = x, x_1, \dotsc, x_{k-1}, x_k = y \in \bR^n$ such that $[x_i]\!\!\! \mod \bZ^n$ belongs to $X_0$ for all $0 \leq i \leq k$ and $[x_{i-1},x_{i}]\!\!\! \mod \bZ^n$ belongs to $X_1$ and has degree one for $1 \leq i \leq k$. 
		\item A $k$-simplex $[x_0, \dotsc, x_k]\!\!\! \mod \bZ^n$ with $k > 1$ belongs to $X_k$ if and only if $[x_i]\!\!\! \mod \bZ^n$ belongs to $X_0$ for $0 \leq i \leq k$ and $[x_{i-1}, x_{i}]\!\!\!\mod \bZ^n$ belongs to $X_1$ for $1 \leq i \leq k$. 
	\end{enumerate}
\end{Proposition}
\begin{proof}
Suppose $X$ admits such a function. Set $I = X_0$, write $I_k \subset I$ for the subset of $i$ with $\deg(i) = k$, and for each $i \in I$ choose some $x_i \in \bR^n$ mapping to the associated point of $T^n$. For each $i, j \in I$ we write $\Exp'_{ij} \subset \bZ^n$ for the set of $m$ such that $[x_j,x_i + m]\!\!\! \mod \bZ^n$ belongs to $X_1$. For each $k$ we write $d^k$ for the $R_n$-valued $I_{k+1} \times I_k$ matrix whose $ij$th entry is $\sum_{m \in \Exp'_{ij}} z^m$. These matrices define a bounded sequence $F^\bullet$ with $F^k = R_n^{\,\oplus I_k}$, and we claim $X = X(F^\bullet)$. 

By construction $X_0 = X(F^\bullet)_0$. To show $X_1 = X(F^\bullet)_1$, it follows from the definitions that we must show that for any $i, j \in I$ the set $\Exp'_{ij}$ coincides with the set $\Exp_{ij}$ associated to~$F^\bullet$. This is true by construction if $\deg(i) = \deg(j) + 1$. More generally, condition~(1) implies that $m \in \Exp'_{ij}$ if and only if there exist $i_0 = j, i_1,  \dotsc, i_{k-1}, i_k = i \in I$ and $\ol{m}_1, \dotsc, \ol{m}_{k-1}, \ol{m}_k = m \in \bZ^n$ such that, for all $1 \leq \ell \leq k$, we have that $[x_{i_{\ell-1}} + \ol{m}_{\ell -1}, x_{i_\ell} + \ol{m}_\ell]\!\!\! \mod \bZ^n$ belongs to $X_1$ and that $\deg(i_\ell) = \deg(i_0) + \ell$ (where $\ol{m}_0 = 0$). The first condition is equivalent to $m_\ell := \ol{m}_\ell - \ol{m}_{\ell-1}$ belonging to $\Exp'_{i_\ell i_{\ell-1}}$ for $1 \leq \ell \leq k$ (here we set $\ol{m}_0 = 0$). But having established the degree one case, the second condition implies $\Exp'_{i_\ell i_{\ell-1}} = \Exp_{i_\ell i_{\ell-1}}$ for $1 \leq \ell \leq k$. Since $m = \ol{m}_k = \sum_{\ell = 1}^k m_\ell$, these conditions are then equivalent by inspection to the conditions in Definition \ref{def:Mij} that characterize when $m \in \Exp_{ij}$. 

To show that $X_k = X(F^\bullet)_k$ for $k > 1$, note that for any linear $k$-simplex $\sigma_k$ with vertices in $X_0$ we have $\sigma_k = [x_{i_0}, x_{i_1} + m'_1, \dotsc, x_{i_k} + m'_k]\!\!\! \mod \bZ^n$ for some uniquely determined $i_0, \dotsc, i_k \in I$ and $m'_1, \dotsc, m'_k \in \bZ^n$. By the previous paragraph $\sigma_k \in X_k$ if and only if $m_\ell = m'_{\ell} - m'_{\ell-1}$ belongs to $\Exp_{i_\ell i_{\ell-1}}$ for $1 \leq \ell \leq k$ (where $m'_0 =0$). But since $m'_\ell = \sum_{j = 1}^\ell m_j$, this is equivalent to the condition in Definition \ref{def:XFbullet} that characterizes when $\sigma_k \in X(F^\bullet)_k$. 
\end{proof}

\begin{Remark}\label{rem:characterization}
	It is natural to further ask which $X$ are of the form $X(F^\bullet)$ when we also require $F^\bullet$ to be a cochain complex. It is easy to see that this is a nontrivial condition on $X$. For example, suppose $X$ consists of the faces of a single 2-simplex $[x_0, x_1, x_2] \!\! \mod \bZ^n$. Then $X \cong X(F^\bullet)$ exactly when $F^\bullet$ is a three-term sequence
	$$ \cdots \to 0 \to R_n \xrightarrow{c_{m_1} z^{m_1}} R_n \xrightarrow{c_{m_2} z^{m_2}} R_n \to 0 \to \cdots $$
	where $m_1, m_2 \in \bZ^n$ and $c_{m_1}, c_{m_2} \in \bC^\times$. But such a sequence is never a cochain complex. In general, we are asking that the coefficients in the reconstructed complex $F^\bullet$ admit a solution to the equations entailed by $d^2 = 0$, and we do not know a combinatorial condition on $X$ that characterizes when this is possible. 
\end{Remark}

\section{Mirror complexes}\label{sec:sheaves}

In this section we associate a complex $C^\bullet(F^\bullet)$ of constructible sheaves on $T^n$ to a bounded complex $F^\bullet$ of free $R_n$-modules (Definition \ref{def:sheafonX}). It is a sum of pushforwards of constant sheaves on certain contractible subsets $S_i \subset \bR^n$ indexed by the rank-one summands of $F^\bullet$. Their images in $T^n$ cover $T(F^\bullet)$, and in particular the combined support of the terms of $C^\bullet(F^\bullet)$ is exactly $T(F^\bullet)$.   These $S_i$ admit an inductive characterization, which in a sense explains the direct definition of $T(F^\bullet)$ in the previous section (Proposition~\ref{prop:SiofF}). Our main result is that the complex $F^\bullet$ is mirror to $C^\bullet(F^\bullet)$ in the relevant sheaf-theoretic sense (Theorem \ref{thm:mirrorcomplex}). 

We retain the notation and conventions of the previous section. In particular, $F^\bullet$ is a bounded sequence of finite-rank free $R_n$-modules, $I$ is an index set for a homogeneous basis, and we have a subset $\Exp_{ij} \subset \bZ^n$ for all $i,j \in I$. 

\begin{Definition}\label{def:SiofF}
	We write $\{S_i(F^\bullet)\}_{i \in I}$ for the following collection of subsets of $\bR^n$. First we choose a point $x_i \in \bR^n$ for all $i \in I$. Then we set
	\begin{equation}\label{eq:SiofFdef}
	S_i(F^\bullet) := \bigcup_{k \geq 0}\left(\bigcup_{i_1, \dotsc, i_k \in I} \left(\bigcup_{ \substack{(m_1,\dotsc,m_k) \in \\ \Exp_{i_1 i} \times \cdots \times \Exp_{i_k i_{k-1}}}} [x_{i}, x_{i_1} + m_1, \dotsc, x_{i_k} + \sum_{\ell = 1}^{k} m_\ell]\right)\right)
	\end{equation}
	the $k = 0$ subset on the right being interpreted as the point $[x_i]$. We write $S_i$ for $S_i(F^\bullet)$ when~$F^\bullet$ is understood. 
\end{Definition}

\begin{Proposition}
The union of the images of the $S_i$ in $T^n$ is exactly $T(F^\bullet)$. 
\end{Proposition}
\begin{proof}
The linear simplices appearing in Definition \ref{def:SiofF} are exactly those appearing in Definition~\ref{def:XFbullet}, now grouped by initial vertex rather than dimension. 
\end{proof}

We have the following inductive characterization of these subsets. As a topological space, it says that $S_i$ is the cone over a space glued out of $\Exp_{ji}$ copies of each $S_j$ (or is a point if each $\Exp_{ji}$ is empty). This is then mapped to $\bR^n$ by sending the cone point to $x_i$, sending the copy of $S_j$ labeled by $m \in \Exp_{ji}$ to $S_j + m$, and then linearly interpolating. To state this formally we use the notation 
$$[x,S] = \{t x + (1-t)y\, |\, y \in S, t \in [0,1]\}$$
for a point $x \in \bR^n$ and a nonempty subset $S \subset \bR^n$, setting $[x,S] = \{x\}$ if $S$ is empty.  Note that since $\Exp_{ji}$ is empty when $\deg(j) \leq \deg(i)$, this description uniquely determines the $S_i$ by descending induction on $\deg(i)$. One can also check that (\ref{eq:SiofF}) is redundant in the following sense: one obtains the same subset if the union is only over $j$ with $\deg(j) = \deg(i) + 1$.  

\begin{Proposition}\label{prop:SiofF}
	The $S_i$ are uniquely characterized by the formula 
	\begin{equation}\label{eq:SiofF} S_i = \bigcup_{j \in I} \left(\bigcup_{m \in \Exp_{ji}} [x_i,S_j + m ]\right),
	\end{equation}
	if $E_{ji} \neq \varnothing$ for some $j$ and $S_i = \{x_i\}$ otherwise.
	In particular, the $S_i$ are contractible, compact, subanalytic subsets such that $S_j + m \subset S_i$ for all $i, j \in I$ and $m \in \Exp_{ji}$. 
\end{Proposition}
\begin{proof}
When $\Exp_{ji}$ is empty for all $j \in I$, hence in particular when $\deg(i) = 0$, the stated value $S_i = \{x_i\}$ agrees with Definition \ref{def:SiofF}.  Now suppose $\Exp_{ji}$ is nonempty for some $j \in I$ and that the claim is true for all $j \in I$ with $\deg(j) > \deg(i)$. Since $\Exp_{ji}$ can only be nonempty if $\deg(j) > \deg(i)$, the right-hand side of (\ref{eq:SiofF}) is then the union over $j \in I$, $m \in \Exp_{ji}$, $k \geq 0$, $i_1, \dotsc, i_k \in I$, and $(m_1,\dotsc,m_k) \in \Exp_{i_1 j} \times \cdots \times \Exp_{i_k i_{k-1}}$ of the subsets
\begin{align*}
	[x_i,[x_{j}, x_{i_1} + m_1, \dotsc, x_{i_k} + \sum_{\ell = 1}^{k} m_\ell] + m ]. 
\end{align*}
But relabeling indices as $k' = k + 1$, $i'_1 = j$, $m'_1 = m$, and $i'_\ell = i_{\ell - 1}$, $m'_\ell = m_{\ell-1}$ for $2 \leq \ell \leq k'$, this is the same as the union over $k' \geq 1$, $i'_1, \dotsc, i'_{k'} \in I$, and $(m'_1,\dotsc,m'_{k'}) \in \Exp_{i'_1 i} \times \cdots \times \Exp_{i'_{k'} i'_{k'-1}}$ of the subsets 
\begin{align*}
	[x_{i}, x_{i'_1} + m'_1, \dotsc, x_{i'_{k'}} + \sum_{\ell = 1}^{k'} m'_\ell]. 
\end{align*}
This is the same union appearing in (\ref{eq:SiofFdef}), since the hypothesis that $\Exp_{ji}$ is nonempty for some $j \in I$ makes the $k=0$ term in (\ref{eq:SiofFdef}) redundant. Finally, contractibility is evident since (\ref{eq:SiofF}) deformation retracts onto $\{x_i\}$, and compactness and constructibility are also evident since (\ref{eq:SiofF}) is a finite union of linear simplices. 
\end{proof}

For the following definition, we package the coefficients in the differentials of $F^\bullet$ as follows. For any $i,j \in I$ with $\deg(i) = \deg(j) + 1$, we define $c_{ijm}\in \bC$ for $m \in \Exp_{ij}$ by the condition that the $ij$-entry of the differential $d^k$ is $\sum_{m \in {\Exp_{ij}}} c_{ijm} z^m$. We also recall our convention that given a locally closed subanalytic set $i: X \into M$, we write $\bC_X$ and $\omega_X$ for $i_! \bC_X$ and $i_! \omega_X$ when $i_!$ is clear from context. Note that if $X$ is closed then $i_* \cong i_!$ and $\omega_X \cong \bD_M \bC_X$. 

\begin{Definition}\label{def:sheafonX}
We define a sequence $C^\bullet(F^\bullet)$ of constructible sheaves on $T^n$ as follows. For $i,j \in I$ and $m \in \Exp_{ij}$, write $\phi^m_{ij}$ for the composition of the canonical maps $\pi_* \bC_{S_j} \to \pi_* \bC_{S_i + m}$ and $\pi_* \bC_{S_i + m} \congto \pi_* \bC_{S_i}$. We then set $$ C^\bullet(F^\bullet) = \cdots \to \oplus_{i \in I_k} \pi_* \bC_{S_i} \xrightarrow{d^k} \oplus_{i \in I_{k+1}} \pi_* \bC_{S_i} \to \cdots, $$
where the $ij$-entry of $d^k$ is $\sum_{m \in \Exp_{ij}} c_{ijm} \phi^m_{ij}$. 
\end{Definition}

Next we show that mirror symmetry recovers $F^\bullet$ from $C^\bullet(F^\bullet)$. As discussed in Section~\ref{sec:torisheaves}, the mirror $R_n$-module of a sheaf $\cF\in \Sh(T^n)$ is its image under 
$$ \Hom(\pi_! \omega_{\bR^n},\wr(-)): \Sh(T^n) \to \Mod_{R_n},$$
where we use the isomorphism $\End(\pi_! \omega_{\bR^n}) \cong R_n$ of Proposition \ref{prop:endcomp}. 

In stating the following result, we allow for generalizations of $C^\bullet(F^\bullet)$ defined using other subsets $S_i \subset \bR^n$ than those considered above. This is because the proof only depends on the properties established in Proposition \ref{prop:SiofF} rather than the specific details of Definition \ref{def:SiofF}. Indeed, Definition \ref{def:SiofF} (hence implicitly Definition \ref{def:XFbullet}) should be understood as engineering the most natural ``minimal'' subsets with these properties. 

\begin{Theorem}\label{thm:mirrorcomplex}
	Suppose for each $i \in I$ we have a locally closed subanalytic set $S_i \subset \bR^n$ which is compact and contractible. For any $i, j \in I$ and $m \in \Exp_{ij}$ assume that $S_i + m \subset S_j$, and write $\phi^m_{ij}$ for the composition of the canonical maps $\pi_* \bC_{S_j} \to \pi_* \bC_{S_i + m}$ and $\pi_* \bC_{S_i + m} \congto \pi_* \bC_{S_i}$.  Set
	$$ C^\bullet = \cdots \to \oplus_{i \in I_k} \pi_* \bC_{S_i} \xrightarrow{d^k} \oplus_{i \in I_{k+1}} \pi_* \bC_{S_i}  \to \cdots, $$
	where the $ij$-entry of $d^k$ is $\sum_{m \in \Exp_{ij}} c_{ijm} \phi^m_{ij}$. Then there is a canonical isomorphism 
	\begin{equation}\label{eq:mainiso}
		\Hom(\pi_!\omega_{\bR^n}, \wr C^\bullet) \cong F^\bullet
		\end{equation}
	of sequences of $R_n$-modules. In particular, this holds for $C^\bullet = C^\bullet(F^\bullet)$. If $F^\bullet$ is a cochain complex, then so is $C^\bullet$. 
\end{Theorem}

The theorem will result from the following three lemmas. We begin by noting that since $\wr$ is a left adjoint, it is not immediately clear that morphisms into $\wr C^\bullet$ are computable. Our first lemma gives a dual formulation of this computation that will turn out to be more accessible. Here we write $\cF \mapsto \shortminus \cF$ for the autoequivalence of $\Sh(T^n)$ given by pullback along $\theta \mapsto \shortminus \theta$, and recall that $\Sh^b(T^n) \subset \Sh(T^n)$ denotes the subcategory of constructible sheaves. 

\begin{Lemma}\label{lem:duality}
There is a canonical isomorphism 
\begin{equation}\label{eq:duality} \Hom(\pi_! \omega_{\bR^n},\wr(-)) \cong \Hom(\shortminus \bD(-), \pi_! \omega_{\bR^n}) 
	\end{equation}
of functors $\Sh^b(T^n) \to \Mod_{R_n}$. 
\end{Lemma}
\begin{proof}
Recall that $\wr$ takes $\Sh^b(T^n)$ to the subcategory $\Loc^c(T^n)$ of compact objects in $\Loc(T^n)$ \cite[Lem. 5.2]{Kuo23}. Also recall that convolution restricts to a monoidal structure on $\Loc(T^n)$ such that $\pi_! \omega_{\bR^n}$ is the unit and such that $\wr$ is monoidal (Proposition \ref{prop:monoidalwrapping}). It follows that compact objects of $\Loc(T^n)$ are dualizable, since perfect $R_n$-modules are dualizable and $\Hom(\pi_! \omega_{\bR^n},-)$ restricts to a monoidal equivalence between these (Proposition \ref{prop:monoidalmirror}). Observe now that the functor $(-)^\vee$ of taking duals with respect to convolution in $\Loc^c(T^n)$ fits into the following diagram. 
\begin{equation*}
	\begin{tikzpicture}
		[baseline=(current  bounding  box.center),thick,>=\arrtip]
		\node (aa) at (0,0) {$\Sh^b(T^n)$};
		\node (ab) at (3.5,0) {$\Loc^c(T^n)$};
		\node (c) at (7,-.75) {$\Mod_{R_n}$};
		\node (ba) at (0,-1.5) {$\Sh^b(T^n)^{\op}$};
		\node (bb) at (3.5,-1.5) {$\Loc^c(T^n)^{\op}$};
		\draw[->] (aa) to node[above] {$\wr $} (ab);
		\draw[->] (ab) to node[above right,pos=.3] {$\Hom(\pi_! \omega_{\bR^n},-) $} (c);
		\draw[->] (ba) to node[above] {$\wr $} (bb);
		\draw[->] (bb) to node[below right,pos=.3] {$\Hom(-, \pi_! \omega_{\bR^n}) $} (c);
		\draw[->] (aa) to node[left, pos=.45] {$\shortminus \bD $} node[right] {\raisebox{.5ex}{\rotatebox{90}{$\sim$}}} (ba);
		\draw[->] (ab) to node[left] {\raisebox{.5ex}{\rotatebox{90}{$\sim$}}} node[right, pos=.45]  {$(-)^\vee $} (bb);
	\end{tikzpicture}
\end{equation*}
The left square exists since $\shortminus \bD$ is the functor of taking duals with respect to convolution in $\Sh^b(T^n)$ (Proposition \ref{prop:rigidity}), and since a monoidal functor preserves dualizable objects and thus intertwines the duality functors if its source and target are rigid. The right triangle exists since $(\pi_! \omega_{\bR^n})^\vee \cong \pi_! \omega_{\bR^n}$ by virtue of $\pi_! \omega_{\bR^n}$ being the monoidal unit in $\Loc^c(T^n)$. The claim now follows by comparing the outer compositions around this diagram. 
\end{proof}

\begin{Lemma}\label{lem: wrapping-of-contractbles1}
	Let $S \subset \bR^n$ be a locally closed subanalytic set and $v_S: \omega_S \to \omega_{\bR^n}$ the canonical map. If $S$ is contractible, then $W v_S: W \omega_S \to \omega_{\bR^n}$ is an isomorphism. If additionally $S$ is compact, then we have an isomorphism $\shortminus \bD \pi_* \bC_{\shortminus S} \congto \pi_!\omega_S$. Its composition with $\pi_!v_S: \pi_! \omega_S \to \pi_! \omega_{\bR^n}$ gives rise to a map $\psi_{\shortminus S}: \pi_! \omega_{\bR^n} \to W \pi_* \bC_{\shortminus S}$ under (\ref{eq:duality}), and $\psi_{\shortminus S}$ is also an isomorphism. 
	In particular, the map
	\begin{equation}\label{eq:CStrivialization}
		R_n \cong \End(\pi_! \omega_{\bR^n}) \to \Hom(\pi_! \omega_{\bR^n},\wr \pi_* \bC_{\shortminus S})
	\end{equation}
	given by composition with $\psi_{\shortminus S}$ is an isomorphism. 
\end{Lemma}
\begin{proof}
Write $p: \bR^n \to \pt$ for the projection. Since $\bR^n$ is contractible, $p_!$ restricts to an equivalence between $\Loc(\bR^n)$ and $\Sh(\pt) \cong \Mod_\bC$. It thus suffices to show $p_! W v_S$ is an isomorphism. By Lemma \ref{lem: compatibility-pushforward-local-system} this is equivalent to $W p_! v_{S} \cong p_! v_S$ being an isomorphism. But $p_! v_S$ is the natural map from the homology of $S$ to the homology of $\bR^n$, which is an isomorphism since both are contractible. 

If additionally $S$ is compact, then the restriction of $\pi$ to $S = \supp(\omega_S)$ is proper, hence $\pi_! \omega_{S}$ is isomorphic to $\pi_* \omega_{S}$ and is constructible \cite[Prop. 8.4.8]{KS94}. It further follows that $\pi_* \bC_{\shortminus S}$ is constructible and that we have the desired isomorphism $\shortminus \bD \pi_* \bC_{\shortminus S} \cong \pi_!\omega_S$. By definition, $\psi_{\shortminus S}$ is the composition of the map $(W \pi_! v_S)^\vee: (\pi_! \omega_{\bR^n})^\vee \to (W \pi_! \omega_S)^\vee$ with the isomorphisms $\pi_! \omega_{\bR^n} \cong (\pi_! \omega_{\bR^n})^\vee$ and $(W \pi_! \omega_S)^\vee \cong \shortminus W \bD \pi_! \omega_S \cong W \pi_* \bC_{\shortminus S}$. But $W \pi_! v_S \cong \pi_! W v_S$ by Lemma~\ref{lem: compatibility-pushforward-local-system}, hence $W \pi_! v_S$ and thus $\psi_{\shortminus S}$ are isomorphisms since $W v_S$ is. 
\end{proof}

\begin{Lemma} \label{lem: wrapping-of-contractbles2}
	Let $S$ and $S'$ be compact subanalytic subsets of $\bR^n$, and suppose that $S' + m \subset S$ for some $m \in \bZ^n$. Write $\phi^m$ for the composition of the canonical maps $\pi_* \bC_{S} \to \pi_* \bC_{S' + m}$ and $\pi_* \bC_{S' + m} \congto \pi_* \bC_{S'}$. Then there is a canonical diagram
	\begin{equation}\label{eq: wrapping-of-contractbles}
	\begin{tikzpicture}
		[baseline=(current  bounding  box.center),thick,>=\arrtip]
		\node (a) at (0,0) {$\pi_! \omega_{\bR^n}$};
		\node (b) at (4.5,0) {$\pi_! \omega_{\bR^n}$};
		\node (c) at (0,-1.5) {$\wr \pi_* \bC_{S}$};
		\node (d) at (4.5,-1.5) {$\wr \pi_* \bC_{S'}$,};
		\draw[->] (a) to node[above] {$z^m $} (b);
		\draw[->] (b) to node[right] {$\psi_{S'} $} (d);
		\draw[->] (a) to node[left] {$\psi_{S} $}(c);
		\draw[->] (c) to node[above] {$\wr \phi^m $} (d);
	\end{tikzpicture}
	\end{equation}
	where $\psi_{S}$ and $\psi_{S'}$ are the maps of Lemma~\ref{lem: wrapping-of-contractbles1}. 
\end{Lemma}

\begin{proof}
Consider the following diagram.
\begin{equation}\label{eq:bigsquare}
	\begin{tikzpicture}
		[baseline=(current  bounding  box.center),thick,>=\arrtip]
		\node (aa) at (0,0) {$\omega_{\shortminus S'}$};
		\node (ab) at (3.5,0) {$\tau_{\shortminus m}^! \omega_{\shortminus S' \shortminus m}$};
		\node (ac) at (7,0) {$\pi^! \pi_! \omega_{\shortminus S' \shortminus m}$};
		\node (ad) at (10.5,0) {$\pi^! \pi_! \omega_{\shortminus S}$};
		\node (ba) at (0,-1.5) {$\omega_{\bR^n}$};
		\node (bb) at (3.5,-1.5) {$\tau_{\shortminus m}^! \omega_{\bR^n}$};
		\node (bc) at (10.5,-1.5) {$\pi^! \pi_! \omega_{\bR^n}$.};
		\draw[->] (aa) to node[above] {$ $} (ab);
		\draw[->] (ab) to node[above] {$ $} (ac);
		\draw[->] (ac) to node[above] {$ $} (ad);
		\draw[->] (ba) to node[above] {$ $} (bb);
		\draw[->] (bb) to node[above] {$ $} (bc);
		\draw[->] (aa) to node[left] {$v_{\shortminus S'} $}(ba);
		\draw[->] (ab) to node[right] {$\tau_{\shortminus m}^! v_{\shortminus S'\shortminus m} $} (bb);
		\draw[->] (ac) to node[left] {$\pi^! \pi_! v_{\shortminus S'\shortminus m} $} (bc);
		\draw[->] (ad) to node[right] {$\pi^! \pi_! v_{\shortminus S} $} (bc);
	\end{tikzpicture}
\end{equation}
Here the lower left horizontal map comes from $\tau_{\shortminus m}^! \pt^! = \pt^!$, the lower right and top middle maps come from the natural transformation $\tau_{\shortminus m}^! \cong \tau_{m!} \to \pi^! \pi_!$ of Proposition \ref{prop:endcomp}, and the top right is the image of the counit map $\omega_{\shortminus S' \shortminus m} \to \omega_{\shortminus S}$ under $\pi^! \pi_!$. 

Unwinding the definitions of $\phi^m$ and $z^m$, the outer square of (\ref{eq:bigsquare}) now gives rise by adjunction to the following diagram. 
\begin{equation*}
	\begin{tikzpicture}
		[baseline=(current  bounding  box.center),thick,>=\arrtip]
		\node (a) at (0,0) {$\pi_! \omega_{\shortminus S'}$};
		\node (b) at (4.5,0) {$\pi_! \omega_{\shortminus S}$};
		\node (c) at (0,-1.5) {$\pi_! \omega_{\bR^n}$};
		\node (d) at (4.5,-1.5) {$\pi_! \omega_{\bR^n}$.};
		\draw[->] (a) to node[above] {$\shortminus \bD \phi^m $} (b);
		\draw[->] (b) to node[right] {$\pi_! v_{\shortminus S} $} (d);
		\draw[->] (a) to node[left] {$\pi_! v_{\shortminus S'} $}(c);
		\draw[->] (c) to node[above] {$z^m $} (d);
	\end{tikzpicture}
\end{equation*}
From this we obtain a diagram of the form (\ref{eq: wrapping-of-contractbles}) by applying Lemma \ref{lem:duality}.  
\end{proof}

\begin{proof}[Proof of Theorem \ref{thm:mirrorcomplex}]
	The isomorphism (\ref{eq:mainiso}) follows immediately from Lemmas \ref{lem: wrapping-of-contractbles1} and~\ref{lem: wrapping-of-contractbles2}. It exists for $C^\bullet = C^\bullet(F^\bullet)$ by Proposition \ref{prop:SiofF}. 
	
	The condition that $F^\bullet$ is a complex (i.e. its differential squares to zero) is equivalent to the condition that for any $i, k \in I$ with $\deg(i) = \deg(k) + 2$, we have
	\begin{equation}\label{eq:diff1} \sum_{j \in I} \sum_{m \in \Exp_{ij}} \sum_{m' \in \Exp_{jk}} c_{ijm}c_{jkm'} z^{m+m'} = 0. \end{equation}
	Similarly, $C^\bullet$ being a complex is equivalent to the condition that for any such $i, k \in I$ we have
	\begin{equation}\label{eq:diff2}  \sum_{j \in I} \sum_{m \in \Exp_{ij}} \sum_{m' \in \Exp_{jk}} c_{ijm}c_{jkm'} \phi_{ij}^m \phi_{jk}^{m'} = 0. \end{equation}
	But note that  $\phi_{ij}^m \phi_{jk}^{m'} = \phi_{ik}^{m + m'}$, where $\phi_{ik}^{m + m'}$ is the composition of the canonical maps $\pi_* \bC_{S_k} \to \pi_* \bC_{S_i + m + m'}$ and $\pi_* \bC_{S_i + m + m'} \congto \pi_* \bC_{S_i}$. 
	It follows that (\ref{eq:diff1}) and (\ref{eq:diff2}) are both equivalent to the condition that for any $m'' \in \bZ^n$, we have
	$$ \sum_{j \in I} \sum_{m \in \Exp_{ij}} \sum_{m' \in \Exp_{jk}} c_{ijm}c_{jkm'}[m + m' = m''] = 0, $$
	where $[m + m' = m'']$ is zero or one depending on whether the equality holds. In particular, $C^\bullet$ is a complex if $F^\bullet$ is. 
\end{proof}

\begin{Example}
Here is a special case of Theorem \ref{thm:mirrorcomplex} from the existing literature. Let $\cF$ be a coherent sheaf on a smooth projective toric variety $X_\Sigma$, and let $\cE^\bullet$ be a locally free resolution of $\cF$ whose terms are sums of anti-effective line bundles. By results of \cite{HHL23,BE24} such a resolution exists, for example, when $\cF$ is the structure sheaf of a toric subvariety. If in addition the line bundles involved are anti-nef, then the nonequivariant coherent-constructible correspondence identifies $\cE^\bullet$ as being mirror to a complex $C^\bullet_{\Delta}$ of the kind appearing in Theorem \ref{thm:mirrorcomplex}, where the $S_i$ are certain polytopes (full-dimensional, if the summands of $\cE^\bullet$ are anti-ample) \cite[Prop. 3.12]{Zho19}. Here mirror refers to the equivalence between $\Coh(X_\Sigma)$ and sheaves whose singular support belongs to a certain Legendrian determined by the fan $\Sigma$. It follows from e.g. \cite[Cor. 9.5]{Kuw20} that the restriction of $\cF$ to $(\bC^\times)^n$ is mirror to $W C^\bullet_{\Delta}$, as Theorem \ref{thm:mirrorcomplex} also asserts in this case. 
\end{Example}

\section{Realizability}\label{sec:branes}

In this section we consider the question of realizing tropical Lagrangian coamoebae as degenerations of coamoebae of smooth Lagrangian branes in $T^* T^n$. As sketched in the introduction, we anticipate that in good cases there exists an isotopy $\{L_t\}_{t > 0}$ of (possibly immersed) Lagrangian branes whose sheaf quantizations converge to $C^\bullet(F^\bullet)$ and whose coamoebae converge to $T(F^\bullet)$ as $t \to 0$. Our main result is that this is the case when $F^\bullet$ is a two-term complex (Theorem \ref{thm:hypersmoothing}). 

In the two-term case $T(F^\bullet)$ is a bipartite graph, and the Lagrangian branes we construct are based on a number of antecedents in the literature \cite{STWZ19,Mat21,Hic19} (see Remark~\ref{rem:otherLags} for more discussion). When $n=2$, their sheaf quantizations were termed alternating sheaves and studied in \cite{STWZ19}. Here we extend this terminology and a number of results of \cite{STWZ19} to higher dimensions. We then consider degenerations of alternating sheaves in which their support retracts onto a bipartite graph. The limit of such a degeneration is what we call a ``reflected local system'' (Proposition \ref{prop:limitaltsheafquant}), a fact we then use to prove Theorem~\ref{thm:hypersmoothing}. 

We begin with some definitions. Throughout we fix an $n$-manifold $M$ and let $\pi$ denote the projections from its cotangent and cosphere bundles.  Below $x_0, \dotsc, x_n$ and $\xi_0, \dotsc, \xi_{n-1}$ are base and fiber coordinates on $T^* \bR^n$, we write $S^{n-2} \subset \bR^{n-1}$ for the unit sphere, and we embed $S^* \bR^n$ into $T^* \bR^n$ as unit covectors. 

\begin{Definition}\label{def:stdev}
The standard Legendrian eversion $\Lambda_{std}\subset S^* \bR^n$ is the image of the Legendrian embedding $\psi: \bR \times S^{n-2}  \to S^* \bR^n$ given by
\begin{gather*}
	\psi^*(x_i) = \begin{cases}
		x_0 & i = 0 \\
		x_0 x_i & i > 0 ,
	\end{cases}
	\quad
	\psi^*(\xi_i) = \begin{cases}
		1/\sqrt{2} & i = 0 \\
		\shortminus x_i /\sqrt{2} & i > 0.
	\end{cases}
\end{gather*}
The front projection $\pi(\Lambda)$ of a Legendrian $\Lambda \subset S^* M$ has an eversion at $x \in M$ if there is a diffeomorphism between neighborhoods of $x \in M$ and $0 \in \bR^n$ which identifies $\Lambda$ with $\Lambda_{std}$.  
\end{Definition}

In words, the standard eversion is one of the two embedded Legendrians whose front projection is the cone $x_0^2 = \sum_{i > 0} x_i^2$. It is distinguished from the other by lifting the smooth part of this cone to conormals pointing towards the $x_0$-axis when $x_0 > 0$ and away from the $x_0$-axis when $x_0 < 0$. 

When $n=2$ we would simply call an eversion a crossing, as it is a point where the immersed curve $\pi(\Lambda)$ crosses itself. But when $n>2$ the term crossing is less apt, while the term eversion captures the fact that $\Lambda_{std}$ changes from lifting a sphere in $\bR^n_{x_0 = t}$ along inward conormals to outward conormals as $t$ changes sign. 

We say a component $U$ of $M \smallsetminus \pi(\Lambda)$ is globally inward (resp. outward) if $\Lambda$ lifts the smooth part of its boundary $\partial \overline{U}$ to inward (resp. outward) conormals. If an eversion $x$ of $\pi(\Lambda)$ is in the closure of $U$, we say that $U$ is locally inward (resp. outward) near $x$ if $\Lambda$ lifts the smooth part of $\partial \overline{U}$ to inward (resp. outward) conormals in a neighborhood of $x$. We note that condition (3) in the following definition is only for convenience --- subsequent results remain true without it, if we modify their statements appropriately. 

\begin{Definition}\label{def:alternatingLeg}
	A Legendrian $\Lambda \subset S^* M$ is alternating if
	\begin{enumerate}
		\item the restriction $\pi: \Lambda \to M$ is an embedding except above the eversions of $\pi(\Lambda)$,
		\item if a component of $M \smallsetminus \pi(\Lambda)$ is locally inward (resp. outward) near an eversion then it is globally inward (resp. outward), 
		\item each globally inward or outward component of $M \smallsetminus \pi(\Lambda)$ is contractible and has an eversion in its closure, and 
		\item $\pi(\Lambda)$ is contained in the closure of the globally inward or outward components of $M \smallsetminus \pi(\Lambda)$. 
	\end{enumerate}
	Here conditions (3) and (4) are not fundamentally essential, but rule out degenerate cases that make subsequent statements and proofs cleaner. If $\Lambda$ is alternating, we say a component of $M \smallsetminus \pi(\Lambda)$ is black, white, or null if it is globally inward, globally outward, or neither. We write $j_w: U_w \into M$ for the union of the white components of $M \smallsetminus \pi(\Lambda)$ and $i_w: \overline{U}_w \into M$ for its closure, similarly for $j_b: U_b \into M$ and $i_b: \overline{U}_b \into M$. If $\Lambda$ is not clear from context, we write $U_w(\Lambda)$, etc. 
\end{Definition}

If $L$ is a subset of $T^* M$, define its asymptotic boundary $\partial_\infty L \subset S^* M$ as follows. Form the fiberwise spherical compactification of $T^* M$, take the closure of $L$, and then intersect it with the fiberwise boundary. Given $\cF \in \Sh(M)$, we write $ss^\infty(\cF) \subset S^* M$ for the asymptotic boundary of its singular support $ss(\cF) \subset T^*M$. 

\begin{Definition}
	A sheaf $\cF \in \Sh_\Lambda(M)$ is alternating relative to an alternating Legendrian $\Lambda$ if $\cF_x \cong 0$ whenever $x$ belongs to a null component of $M \smallsetminus \pi(\Lambda)$. We say $\cF$ has rank $m$ if $\cF_x$ has $m$-dimensional cohomology for all $x \in U_w \cup U_b$.  We write $\Sh^{alt}_\Lambda(M)$ for the full subcategory of such sheaves. 
\end{Definition}

\begin{Definition}
A Lagrangian $L \subset T^* M$ is alternating if (1) it is exact, embedded, and eventually conical, (2) $\Lambda:= \partial_\infty L$ is an alternating Legendrian, and (3) $\pi: L \to M$ is one-to-one above $U_w$ and $U_b$ and zero-to-one outside of $\ol{U}_w \cup \ol{U}_b$. 
\end{Definition}

When $n=2$, the above definitions specialize to corresponding definitions in \cite{STWZ19} with minor variations. Most visibly, what we call an alternating Lagrangian here was called a conjugate Lagrangian there. Similarly, the next few results are higher-dimensional extensions of corresponding results in \cite{STWZ19}. The main differences in the general case are that the front projection $\Lambda \to M$ of an alternating Legendrian is no longer an immersion when $n > 2$, and that the local computations of \cite{STZ17} can no longer be directly applied. 

\begin{Proposition}[{cf. \cite[Prop. 4.9]{STWZ19}}]\label{prop:altfilling}
Every alternating Legendrian admits a filling by an alternating Lagrangian. 
\end{Proposition}

The proof uses the following local model, where we again write $x_0, \dotsc, x_{n-1}$ and $\xi_0, \dotsc, \xi_{n-1}$ for base and fiber coordinates on $T^* \bR^n$, and where we write $D^{n-1} \subset \bR^{n-1}$ for the open unit disk. Here and throughout the section we let it be implicit that formulas with removable singularities mean the underlying smooth function. 

\begin{Lemma}\label{lem:localmodel}
Let $f = x_0 (1 - \sum_{i > 0}x_i^2)^{1/2}$ and $g = \sgn(x_0) (x_0^2 - \sum_{i > 0} x_i^2)^{1/2}$. Then the map $\phi: \bR \times D^{n-1} \to T^* \bR^n$ given by
	\begin{gather*}
		\phi^*(x_i) = \begin{cases}
			x_0 & i = 0 \\
			x_0 x_i & i > 0 ,
		\end{cases}
		\quad
		\phi^*(\xi_i) = \begin{cases}
			\frac{x_0}{f} & i = 0 \\
			\shortminus \frac{x_0 x_i}{f} & i > 0
		\end{cases}
	\end{gather*}
	is an exact Lagrangian embedding such that $f$ is a primitive of $\phi^*(\lambda)$. Its image is the closure of the graph of $dg$ above the locus where $x_0^2 - \sum_{i > 0} x_i^2 > 0$. 
\end{Lemma}
\begin{proof}
	Direct computation shows that $\phi^*(\lambda) = df$ and that $\phi^*(g) = f$ when $x_0 \neq 0$. The image of $\phi$ is closed since $1/f$ goes to infinity along the boundary of $D^{n-1}$, and the last claim follows. 
\end{proof}

\begin{proof}[Proof of Proposition \ref{prop:altfilling}]
Let $\Lambda \subset S^* M$ be an alternating Legendrian and $\{x_i\}$ its eversions. By hypothesis each $x_i$ has a neighborhood $U_i$ which is diffeomorphic to a neighborhood of $0 \in \bR^n$ in a way that identifies $\Lambda$ with the standard Legendrian eversion. 

Let $L \subset T^* M$ be the closure of the graph of $dh$, where $h$ is a smooth function on $U_{w} \cup U_{b}$ with the following properties. First, the restriction of $h$ to each $U_i$ is identified with the function $g$ of Lemma \ref{lem:localmodel} on the corresponding neighborhood of $0 \in \bR^n$. 
Second, near any point on the smooth locus of the boundary of $U_{w}$ (resp. $U_{b}$) we can choose local coordinates $y_0, \dotsc, y_{n-1}$ such that $U_{w}$ (resp. $U_{b}$) is given by $y_0 > 0$, and such that $h = \shortminus y_0^{1/2}$ (resp. $h = y_0^{1/2}$). Note that in the local model the function $g$ satisfies this condition. 

By Lemma \ref{lem:localmodel}, $L$ is an embedded, exact Lagrangian asymptotic to $\Lambda$. The claim now follows by choosing a Hamiltonian flow near the fiberwise boundary that makes $L$ eventually conical, while preserving its asymptotics and the closure of its projection to $M$. 
\end{proof}

\begin{Remark}\label{rem:otherLags}
	Aside from \cite{STWZ19}, another antecedent case of the construction of Proposition \ref{prop:altfilling} is the Lagrangian pair of pants of \cite[Def. 3.7]{Mat21}. In the setting of Theorem~\ref{thm:hypersmoothing} below, it corresponds to $F^\bullet = R_n \xrightarrow{f} R_n$ for $f$ a polynomial of degree one. The two $n$-simplices appearing in \cite{Mat21} are what we call the white and black components of $T^n \smallsetminus \pi(\Lambda)$ for some alternating Legendrian $\Lambda$ (strictly speaking, for this we should relax the smoothness hypotheses included in our definitions). In particular, the local model of Lemma \ref{lem:localmodel} is a variant of the local model used in \cite[Sec. 3.1]{Mat21}.  
	
	The dimer Lagrangians introduced in \cite[Sec. 6.1]{Hic19} generalize the pair of pants, and provide another instance of Proposition \ref{prop:altfilling} already in the literature. They correspond to the case where the white and black components of $T^n \smallsetminus \pi(\Lambda)$ are polyhedra meeting at their vertices and satisfying certain combinatorial conditions. For example, suppose the Newton polytope of $f \in R_n$ has no interior vertices, and express it as the intersection $\cap_{\xi} \bR^n_{\xi \geq 0}$ of the nonnegative loci of a finite collection of affine-linear functionals. In the setting of Theorem~\ref{thm:hypersmoothing} we may arrange for one of the relevant alternating Legendrians to have front projection contained in the image of $\cup_\xi \bR^n_{\xi = \ep}$ under $\pi: \bR^n \to T^n$ for some $\epsilon$ (again up to smoothing). The filling of Proposition \ref{prop:altfilling} is then a dimer Lagrangian, and has coamoeba equal to the closure of a union of two components of $T^n \smallsetminus \pi(\cup_\xi \bR^n_{\xi = \ep})$. 
	
	In general, however, the alternating Lagrangians appearing in Theorem \ref{thm:hypersmoothing} cannot be isotoped to have polyhedral coamoebae. For example, this does not seem to be possible for the hypersurface resolution $R_n \xrightarrow{f} R_n$ once the Newton polytope of $f$ has interior vertices. 
\end{Remark}

\begin{Proposition}[{cf. \cite[Prop. 4.14]{STWZ19}}]\label{prop:explicitaltsheaf} Let $\Lambda \subset S^* M$ be a connected alternating Legendrian and $\cF \in \Sh_\Lambda(M)$. Then $\cF$ is an alternating sheaf of rank one if and only if we may shift its degree so that there exists a triangle
	\begin{equation}\label{eq:alttriangle} j_{w!}\omega_{U_w}[\shortminus 1] \to \cF \to i_{b*}\bC_{\ol{U}_b}.\end{equation} 
	Conversely, an extension $\cF \in \Sh(M)$ of this form is an alternating sheaf if and only if it belongs to $\Sh_\Lambda(M)$. 
\end{Proposition}
\begin{proof}
	The if direction of the first claim is trivial, since $ss^\infty(\cF) \subset \Lambda$ by assumption. For the only if direction, write $j: M \smallsetminus \overline{U}_b \into M$ for the complement of $\ol{U}_b$. We claim the triangle 
	$$  j_{!}j^! \cF \to \cF \to i_{b*} i^*_b \cF $$
	is of the form (\ref{eq:alttriangle}). 
	
	Writing $i: (M \smallsetminus \overline{U}_b) \smallsetminus U_w \into M \smallsetminus \overline{U}_b$ for the inclusion, the cone over the natural map $j_{w!}j^!_w \cF \to j_{!}j^! \cF$ is $j_! i_* i^* j^! \cF$. If $x \in (M \smallsetminus \overline{U}_b) \smallsetminus U_w$, either $x$ is in a null component of $M \smallsetminus \pi(\Lambda)$, or $x$ is on a smooth point of $\pi(\Lambda)$ that $\Lambda$ lifts to a covector pointing into a null component. Either way  $\cF$ being alternating implies $\cF_x \cong 0$, hence $j_! i_* i^* j^! \cF \cong 0$ and $j_{w!}j^!_w \cF \cong j_{!}j^! \cF$. Writing $U_w = \sqcup_\al U_{\al}$ for the decomposition into components, it follows from the contractibility of the $U_{\al}$ that $j^!_w \cF \cong \oplus_\al \omega_{U_{\al}}[k_\al]$ for some $k_\al \in \bZ$. 
	
	Next we claim the map $i_{b*} i^*_b \cF \to j_{b*} j^*_b \cF$ is an isomorphism. Given $x \in M$, it is clear from~$\cF$ being alternating that $(i_{b*} i^*_b \cF)_x \to (j_{b*} j^*_b \cF)_x$ is an isomorphism if  $\pi(\Lambda)$ does not have an eversion at $x$. If it does, let $\xi \in T^*_x M$ be the covector pulled back from $(0,dx_0)$ in a local identification with $\Lambda_{std}$. Up to a shift, the microstalk of $\cF$ at $\xi$ is the cone over the generization map from $\cF_x$ to the stalk of $\cF$ at a nearby point in the interior of the adjacent black region. Since $\xi \notin \Lambda$, this cone is zero and the generization map is an isomorphism. But this generization map is also an isomorphism for $j_{b*} j^*_b \cF$, since near $x$ the closure $\overline{U}_b$ is homeomorphic to a closed cone with vertex $x$. Thus $i_{b*} i^*_b \cF \to j_{b*} j^*_b \cF$ induces an isomorphism of stalks at $x$ since it does so for nearby stalks, and $i_{b*} i^*_b \cF \cong j_{b*} j^*_b \cF$.  Writing $U_b = \sqcup_\be U_{\be}$ for the decomposition into components, it again follows from contractibility that $j^*_b \cF \cong \oplus_\be \bC_{U_{\be}}[k_\be]$ for some $k_\be \in \bZ$. 
	
	It remains to show that for some $k \in \bZ$ we have $k_\al = k \shortminus 1$ for all $\al$ and $k_\be = k$ for all~$\be$. Fix an eversion $x$ of $\pi(\Lambda)$, and let $j_{\al}: U_{\al} \into M$ and $j_{\be}:U_{\be} \into M$ be the unique white and black components containing $x$ in their closures. Note that near $x$ we have
	\begin{align}\label{eq:sheafHomalt}
		\cHom(j_{\be*}\bC_{U_{\be}}, j_{\al!}\omega_{U_{\al}}) \cong \bD(j_{\be*}\bC_{U_{\be}} \otimes \bD j_{\al!}\omega_{U_{\al}}) 
		\cong \bD(j_{\be*}\bC_{U_{\be}} \otimes j_{\al*}\bC_{U_{\al}}) 
		\cong \bC_{x}.
	\end{align}
	It follows that there is a nonzero extension of $j_{\be*}\bC_{U_{\be}}[k_\be]$ by $j_{\al!}\omega_{U_{\al}}[k_\al]$ if and only if $k_\al = k_\be -1$. But $\cF$ restricts to such an extension near $x$, since the singular support at $x$ of the trivial extension is larger than $\Lambda \cap S^*_x M$. The claim now follows since $\Lambda$ is connected, hence for any $\al$, $\be$ the components $U_\al$ and $U_\be$ are connected by a path through $\ol{U}_w \cup \ol{U}_b$. 	
\end{proof}

Next we explain how alternating sheaves and alternating Lagrangians are related by sheaf quantization. Recall that the sheaf of brane data on an exact, eventually conical Lagrangian~$L$ is the image of the Kashiwara-Schapira sheaf $\mu sh_{L'}$ under $L \cong L'$, where $L' \subset S^*(M \times \bR)$ is the Legendrian lift of a lower exact perturbation of $L$ \cite{JT17} (see Section~\ref{sec:sheafquant} for more discussion). Given $\cE \in \mu sh_{L'}(L')$, the sheaf quantization of the brane $(L,\cE)$ is the image of~$\cE$ under the composition 
$$ \mu sh_{L'}(L') \xrightarrow{\sim} \Sh^0_{L'}(M \times \bR) \xrightarrow{p_*} \Sh_{\Lambda}(M). $$
Here the left functor is the inverse of microlocalization, and $\Sh^0_{L'}(M \times \bR) \subset \Sh_{L'}(M \times \bR)$ consists of sheaves vanishing on $M \times (R, \infty)$ for $R \gg 0$.  Note that since $\mu sh_{L'}$ is locally equivalent to the sheaf of local systems on $L$, the brane datum $\cE$ has a well-defined rank. 

\begin{Proposition}[{cf. \cite[Prop. 4.18]{STWZ19}}]\label{prop:altsheafquant}
Let $L$ be an alternating Lagrangian and $\Lambda := \partial_{\infty} L$. Sheaf quantization restricts to an equivalence between the category of Lagrangian branes supported on $L$ and the category of alternating sheaves with $ss^\infty(\cF) \subset \Lambda$. The quantization of a rank-$m$ brane is a rank-$m$ alternating sheaf. 
\end{Proposition}
\begin{proof}
	Fix $\cE \in \mu sh_{L'}(L')$, and let $\cQ$ be the sheaf quantization of $(L,\cE)$. First we show that $\cQ$ is an alternating sheaf. Let $\cF$ be the image of $\cE$ in $\Sh^0_{L'}(M \times \bR)$. Recall that the lower exact perturbation of $L$ may be chosen small enough so that the two coincide except above a neighborhood of $\pi(\Lambda)$. In particular, if $x \in M$ belongs to a null component of $M \smallsetminus \pi(\Lambda)$, we may assume $\{x\} \times \bR$ does not meet $\pi(L') \subset M \times \bR$. It follows that $\cF$ vanishes on $\{x\} \times \bR$, hence $\cQ$ vanishes at $x$, hence $\cQ$ is an alternating sheaf. 
	
	Next we show that $\cQ$ has rank $m$ if $\cE$ does. The latter condition is equivalent to all microstalks of $\cF$ away from the zero-section having $m$-dimensional cohomology. If $x \in M$ belongs to a white or black component of $M \smallsetminus \pi(\Lambda)$, then similarly we may assume $\{x\} \times \bR$ meets $\pi(L')$ at a single point $(x,t)$. The intersection $L' \cap S^*_{(x,t)}(M \times \bR)$ is a single covector which pairs negatively with $dt$. The microstalk of $\cF$ at this covector has $m$-dimensional cohomology since $\cE$ is rank $m$. It follows that the restriction of $\cF$ to $\{x\} \times \bR$ is a sum of $m$ copies of $\bC_{\{x\} \times \bR_{\leq t}}$, shifted to possibly different degrees. But $\bC_{\{x\} \times \bR_{\leq t}}$ has one-dimensional cohomology, hence $\cQ_x$ has $m$-dimensional cohomology, hence $\cQ$ has rank $m$. 
	
	Now choose an open cover $\{U_\al\}$ of $M$ such that (1) each $U_\al$ is contractible, (2) each of $U_\al \cap (\ol{U}_w \cup \ol{U}_b)$ and $U_{\al\be} \cap (\ol{U}_w \cup \ol{U}_b)$, where $U_{\al\be} := U_\al \cap U_\be$, is contractible, and (3) any $U_{\al\be\ga} = U_\al \cap U_\be \cap U_\ga$ is contained in a null component of $M \smallsetminus \pi(\Lambda)$. In particular, $L_\al:= L \cap \pi^{-1}(U_\al)$ is either empty or contractible for all $\al$. In the latter case $\mu sh_{L'}$ is trivializable on the associated $L'_\al \subset L'$, hence $\mu sh_{L'}(L'_\al) \cong \Sh_{L'_\al}^0(U_\al \times \bR)$ is equivalent to $\Mod_\bC$. 
	
	We claim that $p_{\al*}: \Sh^0_{L'_\al}(U_\al \times \bR) \to \Sh^{alt}_{\Lambda_\al}(U_\al)$ is an equivalence for all $\al$, where $\Lambda_\al := \Lambda \cap \pi^{-1}(U_\al)$. If $L_\al$ is empty this follows since both categories are zero. If $L_\al$ is nonempty, pick $x \in U_\al \cap (U_w \cup U_b)$. It follows from quantization preserving rank-one objects that $i_x^* p_{\al*}: \Mod_\bC \cong \Sh^0_{L'_\al}(U_\al \times \bR) \to \Mod_\bC$ is an equivalence. Thus it suffices to show $i_x^*$ restricts to an equivalence $\Sh^{alt}_{\Lambda_\al}(U_\al) \congto \Mod_\bC$. 
	
	If $U_\al$ does not contain an eversion this follows from (2), since  $\Sh^{alt}_{\Lambda_\al}(U_\al)$ is the essential image under either $j_!$ or $j_*$ of local systems on a contractible set ($U_\al \cap U_w$ or $U_\al \cap U_b$, whichever is nonempty). If $U_\al$ contains an eversion, there is an identification $\phi: U_\al \congto V$ with a neighborhood $V$ of $0 \in \bR^n$ which identifies $\Lambda_\al$ with $\Lambda_{std}$. It thus suffices to show $i_{\phi(x)}^*$ restricts to an equivalence $\Sh_{\Lambda_{std}}^{alt}(\bR^n) \congto \Mod_\bC$. But $\Lambda_{std}$ is the total Legendrian of the Legendrian isotopy in $S^*\bR^{n -1}$ induced by Reeb flow of the conormal lift of a sphere around $0 \in \bR^{n-1}$. Restriction to the $x_0$-slice containing $\phi(x)$ is then an equivalence to sheaves extended from a constant sheaf on a ball in $\bR^{n-1}$ \cite[Prop. 3.12]{GKS12}, and the claim follows. 
	
	Given conditions (2) and (3), the same argument shows that $$p_{\al_1 \cdots \al_i*}: \Sh^0_{L'_{\al_1 \cdots \al_i}}(U_{\al_1 \cdots \al_i} \times \bR) \congto \Sh^{alt}_{\Lambda_{\al_1 \cdots \al_i}}(U_{\al_1 \cdots \al_i})$$ is an equivalence for any intersection of the $U_\al$. But $p_*$ and the $p_{\al_1 \cdots \al_i*}$ are all intertwined under the restriction functors. Since $\Sh^0_{L'}(M \times \bR)$ and $\Sh^{alt}_{\Lambda}(M)$ are limits over the \v{C}ech nerves  of the corresponding local categories and restriction functors, it follows that $p_*$ is an equivalence since the $p_{\al_1 \cdots \al_i*}$ are.  
\end{proof}

\begin{Proposition}\label{prop:trivialKSsheaf}
The sheaf of brane data on an alternating Lagrangian $L$ is trivializable, hence the category of Lagrangian branes supported on $L$ is equivalent to $\Loc(L)$. 
\end{Proposition}
\begin{proof}
Again let $\Lambda := \partial_\infty L$ and let $L' \subset S^*(M \times \bR)$ be the Legendrian lift of a lower exact perturbation of $L$. The claim is equivalent to the existence of a rank-one object of $\mu sh_{L'}(L')$ \cite[Prop. 10.2.4]{Gui23}. By Proposition~\ref{prop:altsheafquant}, this is equivalent to the existence of a rank-one alternating sheaf with $ss^\infty(\cF) = \Lambda$. Fixing an open cover $\{U_\al\}$ as in the proof of Proposition~\ref{prop:altsheafquant}, we have seen that such a sheaf exists on each~$U_\al$. But it follows from Proposition \ref{prop:explicitaltsheaf} that, with suitable shifts, these glue together to the desired sheaf on $M$. 
\end{proof}

We now consider certain degenerations of the objects considered above. The following definitions summarize the hypotheses we will want to impose. The only essential point in Definition \ref{def:altlegdegen} is that we have an isotopy of alternating Legendrians whose front projections degenerate to a graph, but having a few extra conditions will simplify the proof of Proposition~\ref{prop:limitaltsheafquant}. We will use in passing that a Legendrian isotopy $\{\Lambda_t\}_{t \in \opint}$ in $S^* M$ (resp. an exact Lagrangian isotopy $\{L_t\}_{t \in \opint}$ in $T^* M$) is equivalent to the data of a Legendrian movie $\Lambda_{\opint} \subset S^*(M \times \opint)$ (resp. Lagrangian movie $L_{\opint} \subset T^*(M \times \opint)$), see \cite[Sec. A.2]{GKS12}. Similarly, given $\cF \in \Sh(M \times \opint)$, we write $\cF_t \in \Sh(M)$ for its pullback to $M \times \{t\}$. 

\begin{Definition}\label{def:altlegdegen}
	A Legendrian isotopy $\Lambda_{\opint} \subset S^*(M \times \opint)$ is an alternating Legendrian degeneration if 
	\begin{enumerate}
		\item $\Lambda_t$ is an alternating Legendrian for all $t$, 
		\item $U_w(\Lambda_t) \subset U_w(\Lambda_{t'})$ and $U_b(\Lambda_t) \subset U_b(\Lambda_{t'})$ for all $t < t'$, 
		\item the set of points where $\pi(\Lambda_t)$ has an eversion is the same for all $t$, 
		\item the Kuratowski limit $\Gamma := \lim_{t \to 0} \pi(\Lambda_t)$ has the structure of an embedded graph such that the following holds: for any $t$, each white or black component of $M \smallsetminus \pi(\Lambda_t)$ contains a single vertex of $\Gamma$, and each vertex is in some such component. 
	\end{enumerate}
\end{Definition}

\begin{Definition}
A family $\cF_{\opint} \in \Sh(M \times \opint)$ is an alternating sheaf degeneration relative to an alternating Legendrian degeneration $\Lambda_{\opint}$ if $ss^\infty(\cF_{\opint}) \subset \Lambda_{\opint}$ and $\cF_t \in \Sh^{alt}_{\Lambda_t}(M)$ for all $t$. We say $\cF_{\opint}$ has rank $m$ if $\cF_t$ has rank $m$ for all $t$. 
\end{Definition}

\begin{Definition}
An exact Lagrangian isotopy $L_{\opint} \subset T^*(M \times \opint)$ is an alternating Lagrangian degeneration if $L_t$ is an alternating Lagrangian for all $t$ and $\Lambda_{\opint} := \partial_\infty L_{\opint}$ is an alternating Legendrian degeneration.
\end{Definition}

\begin{Proposition}\label{prop:altfamilyfillings}
	Every alternating Legendrian degeneration $\Lambda_{\opint}$ admits a filling by an alternating Lagrangian degeneration. If $L_{\opint}$ is such a filling, then sheaf quantization restricts to an equivalence between the category of branes supported on $L_{\opint}$ and the category of alternating sheaf degenerations with $ss^\infty(\cF_{\opint}) \subset \Lambda_{\opint}$. The quantization of a rank-$m$ brane is a rank-$m$ alternating sheaf degeneration. The sheaf of brane data on $L_{\opint}$ is trivializable, hence the category of branes supported on $L_{\opint}$ is equivalent to $\Loc(L_{\opint})$ (hence to $Loc(L_t)$ for any $t \in \opint$). 
\end{Proposition}
\begin{proof}
Fix some $t_0 \in \opint$. By construction there is a family of diffeomorphisms $\{\phi_t\}_{t \in \opint}$ of $M$ such that $\phi_t(\pi(\Lambda_{t_0})) = \pi(\Lambda_t)$, hence such that the induced contactomorphism of $S^* M$ takes $\Lambda_{t_0}$ to $\Lambda_t$. The first claim then follows from the trivial fact that any alternating Lagrangian filling of $\Lambda_{t_0}$ (which exists by Proposition \ref{prop:altfilling}) induces a constant filling of the Legendrian movie of the constant isotopy of $\Lambda_{t_0}$. The remaining claims are similarly equivalent to corresponding claims for the constant isotopy, which follow immediately from Propositions~\ref{prop:altsheafquant} and \ref{prop:trivialKSsheaf}. 
\end{proof}

Our next goal is to compute the limit of an alternating sheaf degeneration in the appropriate sense. Before carrying this out, we give an intrinsic description of the class of sheaves which turn out to arise as such a limit. Given an embedded bipartite graph $\Gamma \subset M$, we write $\Sh_\Gamma(M) \subset \Sh(M)$ for the subcategory of sheaves which are supported on $\Gamma$ and locally constant on the interiors of the edges of $\Gamma$ (i.e. which are constructible with respect to the obvious stratification). We note the condition on isolated vertices in the following definition is only for convenience. 

\begin{Definition}\label{def:reflocsys}
Let $\Gamma \subset M$ be an embedded bipartite graph with no isolated vertices. A sheaf $\cF \in \Sh_\Gamma(M)$ is a reflected local system of rank $m$ if it satisfies the following. 
\begin{enumerate}
	\item Its restriction to the complement of the white vertices is a local system of rank $m$.
	\item If $v$ is a $k$-valent white vertex and $x_1, \dotsc, x_k \in \Gamma$ are nearby points, one on each edge having $v$ as a vertex, then the canonical map $\cF_v \to \cF_{x_1} \oplus \cdots \oplus \cF_{x_k}$ identifies $\cF_v$ with a codimension $m$ subspace whose intersection with $\cF_{x_i}$ is zero for any $1 \leq i \leq k$. 
\end{enumerate}
\end{Definition}

The terminology is motivated as follows. Let $\Gamma_{b \to w}$ be the quiver obtained from $\Gamma$ by orienting its edges from black to white, similarly for $\Gamma_{w \to b}$. There is an equivalence between representations of $\Gamma_{b \to w}$ (resp. $\Gamma_{w \to b}$) and constructible sheaves on $\Gamma$ which are locally constant on the complement of the black (resp. white) vertices. In particular, we may realize local systems on $\Gamma$ as representations of $\Gamma_{b \to w}$ which assign an isomorphism to each edge. On the other hand, there is an equivalence between representations of $\Gamma_{b \to w}$ and of $\Gamma_{w \to b}$ given by applying a reflection functor at each white vertex. Reflected local systems are exactly the representations of $\Gamma_{w \to b}$ which arise from local systems under this equivalence:
$$ \left\{\text{reflected  local systems}
\right\} \subset \Rep \,\Gamma_{w \to b} \cong  \Rep\, \Gamma_{b \to w} \supset \left\{\text{local systems} \right\}
.$$
In particular, the categories of local systems and reflected local systems on $\Gamma$ are canonically equivalent. 

A first indication that reflected local systems are related to alternating sheaves is that they satisfy the following analogue of Proposition \ref{prop:explicitaltsheaf}. Here we adapt the notation of Definition~\ref{def:alternatingLeg} as follows. Suppose we choose a point $x_e$ on each edge~$e$ of an embedded bipartite graph~$\Gamma$. Then we write $j_w: \Gamma_w \into M$ for the union of the components of $\Gamma \smallsetminus \bigcup_e \{x_e\}$ which contain a white vertex and $i_w: \overline{\Gamma}_w \into M$ for its closure, similarly for $j_b: \Gamma_b \into M$ and $i_b: \overline{\Gamma}_b \into M$. 

\begin{Proposition}\label{prop:refloctriangle}
	Let $\Gamma \subset M$ be an embedded bipartite graph with no isolated vertices, $\{x_e\}$ a choice of point on each edge of $\Gamma$, and $\cF \in \Sh_\Gamma(M)$. 
	Then $\cF$ is a reflected local system of rank one if and only if there exists a triangle
	\begin{equation}\label{eq:refloctri2} j_{w!}\omega_{\Gamma_w}[\shortminus 1] \to \cF \to i_{b*}\bC_{\ol{\Gamma}_b}.\end{equation}
	Conversely, an extension $\cF \in \Sh(M)$ of this form is a reflected local system if and only if it belongs to $\Sh_\Gamma(M)$. 
\end{Proposition}
\begin{proof}
Let $v$ be a $k$-valent white vertex of $\Gamma$, and let $\Gamma_v$ be the component of $\Gamma_w$ containing~$v$. If $e_1, \dotsc, e_k$ are the edges having $v$ as a vertex, we write $(v,x_{e_i}) \subset \Gamma$ (resp. $[v,x_{e_i}) \subset \Gamma$ ) for the open (resp. half-open) interval between $v$ and $x_{e_i}$. In $\Sh(\Gamma_v)$ we have triangles 
$$\oplus_{i=1}^k \bC_{(v,x_{e_i})} \to \bC_{\Gamma_v} \to \bC_{\{v\}} \quad\text{and}\quad \omega_{\Gamma_v}[\shortminus 1] \to \oplus_{i = 1}^k \bC_{[v,x_{e_i})} \to \bC_{\{v\}},$$
where the left is trivial and the right follows by applying $\bD$ and rotating. Note that the second map in the right triangle is an epimorphism since it is so on stalks, hence $\omega_{\Gamma_v}[\shortminus 1]$ is concentrated in degree zero. It follows by inspection that $\omega_{\Gamma_v}[\shortminus 1]$ is a rank-one reflected local system (technically $\Gamma_v$ is not bipartite, but nonetheless $\omega_{\Gamma_v}[\shortminus 1]$ satisfies the conditions of Definition~\ref{def:reflocsys}), and that moreover it is the unique such object of $\Sh(\Gamma_v)$ up to isomorphism. 

Now write $j: M \smallsetminus \ol{\Gamma}_b \into M$ for the complement of $\ol{\Gamma}_b$. Given $\cF \in \Sh_\Gamma(M)$, consider the triangle
\begin{equation}\label{eq:refloctri} j_{!}j^! \cF \to \cF \to i_{b*} i^*_b \cF. \end{equation}
Since $\cF$ is supported on $\Gamma$, the natural map $j_{w!} j^*_w \cF \to j_! j^! \cF$ is an isomorphism.  Given the previous paragraph, it now follows that the conditions of Definition~\ref{def:reflocsys} are equivalent to the conditions that $j^*_w \cF$ and $i^*_b \cF$ are respectively isomorphic to $\omega_{\Gamma_w}[\shortminus 1]$ and $j_{b*}\bC_{\Gamma_b}$, hence to the condition that (\ref{eq:refloctri}) is of the form (\ref{eq:refloctri2}). The second claim follows since if $\cF \in \Sh(M)$ is such an extension, then on any edge $e$ it is either locally constant or is locally isomorphic to $j_{w!}\omega_{\Gamma_w}[\shortminus 1] \oplus i_{b*}\bC_{\ol{\Gamma_b}}$, and in the latter case $\cF$ does not belong to $\Sh_\Gamma(M)$. 
\end{proof}

Given $\cF_{\opint} \in \Sh(M \times \opint)$, one makes sense of the limit $\lim_{t \to 0} \cF_t \in \Sh(M)$ as the (real) nearby cycles 
$$ \psi \cF_{\opint} := i^* j_{*} \cF_{\opint},$$
where $i: M \times \{0 \} \into M \times \clint$ and $j: M \times \opint \into M \times \clint$ are the inclusions. We can now state more precisely the relation between alternating sheaves and reflected local systems. 

\begin{Proposition}\label{prop:limitaltsheafquant}
	Let $\cF_{\opint} \in \Sh(M \times \opint)$ be a rank-one alternating sheaf degeneration and $\Lambda_{\opint} := ss^\infty(\cF_{\opint})$. Then, up to a degree shift on each component of $\Gamma$, $\psi \cF_{\opint}$ is a rank-one reflected local system on $\Gamma := \lim_{t \to 0} \pi(\Lambda_t)$. Conversely, any rank-one reflected local system on an embedded bipartite graph without isolated vertices is of this form for some alternating sheaf degeneration.  
\end{Proposition}

\begin{proof}
	We assume $\Gamma$ is connected for convenience, the general case following trivially from this. By condition (3) of Definition \ref{def:altlegdegen}, each edge $e$ of $\Gamma$ contains a unique eversion $x_e$ of $\pi(\Lambda_t)$ for any $t$, this point is independent of~$t$, and every eversion is of this form. By condition (3) of Definition \ref{def:alternatingLeg} there are no isolated vertices of $\Gamma$. 
	Let $j_w: \Gamma_w \into M$, $j_b: \Gamma_b \into M$, $i_w: \ol{\Gamma}_w \into M$, and $i_b: \ol{\Gamma}_b \into M$ be as in the paragraph before Proposition~\ref{prop:refloctriangle}. 
	By the statement of that proposition, to show $\psi \cF_{\opint}$ is a rank-one reflected local system, it suffices to show that there exists a triangle
	\begin{equation}\label{eq:Upsitri} j_{w!}\omega_{\Gamma_w}[\shortminus 1] \to \psi \cF_{\opint} \to i_{b*}\bC_{\ol{\Gamma}_b} \end{equation}
	and that $\psi \cF_{\opint}$ belongs to $\Sh_\Gamma(M)$. 
	
	Define $j'_w: U'_w \into M \times \opint$ so that for any $t$ we have $U'_{w,t} = U_w(\Lambda_t)$, and write $i'_w: \ol{U'}_w \into M \times \opint$ for its closure. Define $j'_b: U'_b \into M \times \opint$ and $i'_b: \ol{U'}_b \into M \times \opint$ similarly. By Proposition \ref{prop:explicitaltsheaf} and \cite[Prop. 3.12]{GKS12} (or by adapting the proof of Proposition~\ref{prop:explicitaltsheaf} to $M \times \opint$) we may shift $\cF_{\opint}$ so that there exists a triangle
	\begin{equation}\label{eq:graphtpsitri} j'_{w!}\omega_{U'_w}[\shortminus 2] \to \cF_{\opint} \to i'_{b*}\bC_{\ol{U'}_b}.\end{equation}
	
	Now note that $\ol{U'}_w$ is topologically a manifold with boundary and $U'_w$ is its interior. In particular, $j'_{w*} \bC_{U'_w} \cong i'_{w*} \bC_{\ol{U'}_w}$. The closure of $\ol{U'}_w$ in $M \times \clint$ decomposes as $\ol{U'}_w \sqcup \ol{\Gamma}_w$, and topologically is a union of closed $(n+1)$-balls minus an open disk in the boundary of each, with $\ol{\Gamma}_w$ embedded into the remaining boundary. Clearly $i'_{w*} \bC_{\ol{U'}_w}$ satisfies the hypothesis of Lemma \ref{lem:verdiervan}, hence we have
	\begin{align*}
		\psi j'_{w!} \omega_{U'_w} \cong \psi \bD j'_{w*} \bC_{U'_w} 
		\cong \psi \bD i'_{w*} \bC_{\ol{U'}_w}
		\cong (\bD \psi i'_{w*} \bC_{\ol{U'}_w}) [1].
	\end{align*}
	Similarly $\ol{U'}_w$ satisfies the hypotheses of Lemma \ref{lem:vancycles}, hence we have
	\begin{align*}
		\bD \psi i'_{w*} \bC_{\ol{U'}_w}
		\cong \bD i_{w*} \bC_{\ol{\Gamma}_w}
		\cong \bD j_{w*} \bC_{\Gamma_w}
		\cong j_{w!} \omega_{\Gamma_w}.
	\end{align*}
The same considerations together with Lemma \ref{lem:vancycles} imply $\psi i'_{b*}\bC_{\ol{U'}_b} \cong i_{b*}\bC_{\ol{\Gamma}_b}$, hence we obtain a triangle (\ref{eq:Upsitri}) by applying $\psi$ to (\ref{eq:graphtpsitri}). 

Given this triangle, it follows that if $\psi \cF_{\opint}$ does not belong to $\Sh_\Gamma(M)$, then at some $x_e$ the extension (\ref{eq:Upsitri}) is trivial. In particular, $ss(\psi \cF_{\opint}) \cap T^*_{x_e}M$ would be strictly larger than $N^*_\Gamma M \cap T^*_{x_e}M$. But $\lim_{t\to 0} \Lambda_t \subset N^*_\Gamma M$ by hypothesis, so this would contradict the fact that $ss^\infty(\psi \cF_{\opint}) \subset \lim_{t\to 0} \Lambda_t$ \cite[Lem. 3.16]{NS20}. Thus $\psi \cF_{\opint}$ is a reflected local system. 

Next we show that every rank-one reflected local system on $\Gamma$ is of this form. In fact, we claim that the map 
\begin{equation}\label{eq:Hommap1} \Hom(i'_{b*}\bC_{\ol{U'}_b},j'_{w!}\omega_{U'_w}[\shortminus 1]) \to \Hom(i_{b*}\bC_{\ol{\Gamma}_b}, j_{w!}\omega_{\Gamma_w}) 
	\end{equation}
induced by $\psi$ is an isomorphism, and that the fiber of $f: i'_{b*}\bC_{\ol{U'}_b} \to j'_{w!}\omega_{U'_w}[\shortminus 1]$ is an alternating sheaf degeneration if and only if the fiber of $\psi(f)$ is a reflected local system. 

Choose a neighborhood $U_e \subset M$ of each $x_e$ as in Definition \ref{def:stdev}.  We first prove the corresponding claims for the induced map 
\begin{equation}\label{eq:Hommap2} \Hom(i'_{b*}\bC_{\ol{U'}_b}|_{U_e \times \opint},  j'_{w!}\omega_{U'_w}[\shortminus 1]|_{U_e \times \opint}) \to \Hom( i_{b*}\bC_{\ol{\Gamma}_b}|_{U_e}, j_{w!}\omega_{\Gamma_w}|_{U_e}). \end{equation}
The calculation (\ref{eq:sheafHomalt}) extends to show both sides are one-dimensional and concentrated in degree zero. Moreover, the fiber of $f_e:  i'_{b*}\bC_{\ol{U'}_b}|_{U_e \times \opint} \to j'_{w!}\omega_{U'_w}[\shortminus 1]|_{U_e \times \opint}$ is an alternating sheaf degeneration (on $U_e \times \opint$) if and only if $f_e \neq 0$, and the fiber of $\psi(f_e)$ is a reflected local system if and only if $\psi(f_e) \neq 0$. But the desired claims are then equivalent to $\psi$ taking alternating sheaf degenerations to reflected local systems, which we have shown. 

On the other hand, (\ref{eq:sheafHomalt}) also extends to show that we have isomorphisms
\begin{gather*}  \Hom(i'_{b*}\bC_{\ol{U'}_b},  j'_{w!}\omega_{U'_w}[\shortminus 1]) \cong \bigoplus_e \Hom(i'_{b*}\bC_{\ol{U'}_b}|_{U_e \times \opint},  j'_{w!}\omega_{U'_w}[\shortminus 1]|_{U_e \times \opint}) \\
\Hom( i_{b*}\bC_{\ol{\Gamma}_b}, j_{w!}\omega_{\Gamma_w}) \cong 	\bigoplus_e \Hom( i_{b*}\bC_{\ol{\Gamma}_b}|_{U_e}, j_{w!}\omega_{\Gamma_w}|_{U_e}).
\end{gather*}
These are compatible with $\psi$, hence the claims about the map (\ref{eq:Hommap1}) follow from the corresponding claims for the maps (\ref{eq:Hommap2}).

The last claim of the Proposition now follows since any embedded bipartite graph $\Gamma$ without isolated vertices is of the form $\lim_{t \to 0} \pi(\Lambda_t)$ for some alternating Legendrian degeneration. Such a degeneration is constructed by (1) taking the boundary of a tubular neighborhood of $\Gamma$, (2) gluing in a copy of the cone $\pi(\Lambda_{std})$ from Definition \ref{def:stdev} along each edge, (3) degenerating this modified boundary onto $\Gamma$, and (4) taking a Legendrian lift. 
	\end{proof}

	We now return to the setting of Sections \ref{sec:restoco} and \ref{sec:sheaves}, fixing a bounded sequence $F^\bullet$ of finite-rank free $R_n$-modules and the resulting constructions $X(F^\bullet)$, $T(F^\bullet)$, and $C^\bullet (F^\bullet)$. We say $T(F^\bullet)$ is embedded if $X(F^\bullet)$ is immersed, the $x_i$ are pairwise distinct modulo $\bZ^n$, and the interiors of distinct simplices have disjoint images in~$T^n$. Note that if the $x_i$ are generic and $F^{k} \cong 0$ for $k \notin \{0, \shortminus 1\}$, then this is automatic if $n > 2$. In the embedded, two-term case $T(F^\bullet)$ is an embedded bipartite graph, and we adopt the convention that its vertices are white or black if their degree is $0$ or $\shortminus 1$, respectively. 
	
	\begin{Theorem}\label{thm:hypersmoothing}
		Suppose that $F^{k} \cong 0$ for $k \notin \{0, \shortminus 1\}$, and that $T(F^\bullet)$ is embedded and has no isolated vertices. Then $C^\bullet(F^\bullet)$ is a reflected local system on $T(F^\bullet)$ placed in degree~$\shortminus 1$. In particular, there exists an isotopy $L_{\opint}$ of exact Lagrangian branes whose coamoebae satisfy $\lim_{t \to 0} C(L_t) = T(F^\bullet)$  and whose sheaf quantization $\cF_{\opint} \in \Sh(T^n \times \opint)$ satisfies $\psi \cF_{\opint} \cong C^\bullet(F^\bullet)$. 
	\end{Theorem}
	\begin{proof}
		We have $S_i = \{x_i\}$ if $\deg(i) = 0$, hence $C^\bullet(F^\bullet)$ is of the form 
		$$ \cdots \to 0 \to \oplus_{i \in I_{\shortminus 1}} \pi_* \bC_{S_i} \xrightarrow{d} \oplus_{i \in I_0} \bC_{\{\theta_i\}} \to 0 \to \cdots,$$
		where $\theta_i := \pi(x_i)$. 
		If $\deg(i) = \shortminus 1$ and $\theta_i$ is a $k$-valent vertex of $T(F^\bullet)$, then $S_i$ is a star-shaped graph with $k$-valent central vertex $x_i$, and $\pi$ maps its edges onto the edges of $T(F^\bullet)$ having~$\theta_i$ as a vertex.
		If $\theta_i$ is connected by an edge to a white vertex $\theta_j$, then by construction the $ji$-entry of $d$ induces a nonzero map of stalks at $\theta_j$. It follows that $d$ is an epimorphism, hence $C^\bullet(F^\bullet)$ is isomorphic to its kernel placed in degree $\shortminus 1$. But it follows by inspection that this kernel is a reflected local system. The remaining claims now follow from Propositions~\ref{prop:altfamilyfillings} and \ref{prop:limitaltsheafquant}. 
	\end{proof}
	
	\begin{Remark}\label{rem:immersedgraphLag}
		We expect the sheaf-theoretic parts of Theorem \ref{thm:hypersmoothing} extend to the case where $X(F^\bullet)$ is merely immersed without major change. However, it is less clear what the correct treatment of immersed sheaf quantization should look like; see \cite[Sec. 1.11.2]{JT17} for a discussion of this. 
	\end{Remark}
	
\section{Minimality}\label{sec:minimality}

In this section we consider the relationship between projective dimensions of $R_n$-modules and support dimensions of mirror constructible sheaves. Recall that in Conjecture~\ref{conj:minimality} we speculated that if $\cF \in \Sh^b(T^n)$ is mirror to $M \in \Mod_{R_n}^\heartsuit$, then $\dim \supp(\cF) \geq \pd(M)$. Here $\dim \supp(\cF)$ means the dimension of a top-dimensional stratum of $\supp(\cF)$, and $\pd(M)$ is the projective dimension of $M$. If $F^\bullet$ is a minimal-length free resolution of $M$, then $\dim \supp (C^\bullet(F^\bullet)) = \pd(M)$ by construction, so the conjecture would provide a minimality condition that distinguishes $T(F^\bullet)$ from the support of an arbitrary mirror of $M$. 

As evidence for the conjecture, we prove here that it holds under the additional hypothesis that $\cF$ is concentrated in a single cohomological degree. 
In fact, the conclusion then holds without assuming $M$ is concentrated in a single degree. For general $M \in \Mod_{R_n}$, we take $\pd(M)$ to mean the length of the smallest interval $[a,b]$ such that $\Ext^i(M,N) \cong 0$ for all $i \notin [a,b]$ and $N \in \Mod_{R_n}^\heartsuit$. 

We caution, however, that $\dim \supp(\cF) \geq \pd(M)$ certainly does not hold for arbitrary mirror pairs $\cF \in \Sh^b(T^n)$ and $M \in \Mod_{R_n}$. For example, $\bC_{\{\pi(0)\}}[k] \oplus \bC_{\{\pi(0)\}}$ has zero-dimensional support, but its mirror $R_n[k] \oplus R_n$ has projective dimension $|k|$. 

\begin{Theorem}\label{thm:minimality}
Suppose that $\cF \in \Sh^b(T^n)$ is concentrated in a single cohomological degree and is mirror to $M \in \Mod_{R_n}$. Then $\dim \supp(\cF) \geq \pd(M)$. 
\end{Theorem}

\begin{proof}
Since projective dimension is invariant under shifts, we may assume $\cF$ is concentrated in degree zero. Choose a triangulation $S$ of $T^n$ such that $\cF$ is $S$-constructible. Let $A$ be an index set for the strata of~$S$. Given a stratum $T^n_\al$, we write $i_\al$ for its inclusion and $d(\al)$ for its dimension. Given $0 \leq k \leq n$, we write $i_k: S_k \to T^n$ (resp. $j_k: U_k \to T^n$) for the inclusion of the union of all strata of dimension $k$ (resp. dimension at least $k$). 

Consider the decreasing filtration of $\cF$ by the sheaves $\cF_{\geq k} := j_{k!}j_k^!(\cF)$. We claim the associated graded sheaf $\cF_{k} := \cF_{\geq k}/ \cF_{> k}$ is of the form $\oplus_{d(\al) = k} i_{\al!} \bC_{T^n_\al}^{m(\al)}$ for some $m: A \to \bN$. To see this, let $i': S_k \to U_k$, $j': U_{k+1} \to U_k$ denote the inclusions, and consider the triangle
$$ j'_! j'^! j_k^!(\cF) \to j_k^!(\cF) \to i'_* i'^* j_k^!(\cF) $$
in $\Sh^b(U_k)$. Since $i'$ is closed and $j_k$ is open the last term is isomorphic to $i'_! i'^* j_k^*(\cF)$. Since $j_{k+1} = j_k \circ j'$ and $i_{k} = j_k \circ i'$, applying $j_{k!}$ then yields a triangle
$$ j_{k+1!}j_{k+1}^!(\cF) \to j_{k!}j_k^!(\cF) \to i_{k!}i_k^*(\cF). $$ 
But since $\cF$ is $S$-constructible and in degree zero, and since the $T^n_\al$ are contractible, $i_{k!}i_k^*(\cF)$ is of the stated form for some $m: A \to \bN$.  

Now consider the associated filtration of $M$ with terms $M_{\geq k} := \Hom(\pi_! \omega_{\bR^n},W \cF_{\geq k})$ and associated graded objects $M_k := M_{\geq k}/M_{> k}$. We claim that $M_k \cong R_n^{m_k}[\shortminus k]$, where $m_k = \sum_{d(\al) = k} m(\al)$. To see this, note that since each $T^n_\al$ is contractible and $\pi$ is a covering map we have $i_\al = \pi \circ i'_\al$ for some lift $i'_\al: T^n_\al \to \bR^n$. We then have
$$ W i_{\al!} \bC_{T^n_\al} 
\cong W \pi_! i'_{\al!} \omega_{T^n_\al}[\shortminus d(\al)]
\cong \pi_! W i'_{\al!} \omega_{T^n_\al}[\shortminus d(\al)]
\cong \pi_! \omega_{\bR^n} [\shortminus d(\al)], $$
where the second isomorphism follows from Lemma \ref{lem: compatibility-pushforward-local-system} and the third from Lemma \ref{lem: wrapping-of-contractbles1}, hence we have $M_k \cong R_n^{m_k}[\shortminus k]$. 

Suppose now that $\dim \supp(\cF) = q$, so that for $k > q$ we have $\cF_k \cong 0$ and thus $m_k = 0$. To show $\pd(M) \leq q$ it suffices to show $\Ext^i(M, N) \cong 0$ unless $\shortminus q \leq i \leq 0$. Certainly this is true if we replace $M$ by $M_{\geq q+1} \cong 0$. But by the previous paragraph, for any $k$ we have a triangle
$$ M_{> k} \to M_{\geq k} \to R_n^{m_k}[\shortminus k].$$ 
By descending induction and the long exact sequence for $\Ext$, it follows that $\Ext^i(M_{\geq k},N) \cong 0$ unless $\shortminus q \leq i \leq 0$, for every $0 \leq k \leq q$. The claim follows since $M_{\geq 0} \cong M$. 
\end{proof}

\bibliographystyle{amsalpha}
\bibliography{bibhigherbipartite}

\end{document}